\documentclass[a4paper,10pt]{article}
\usepackage[utf8]{inputenc}
\usepackage{amsmath}
\usepackage{amssymb}
\usepackage{amsthm}

\theoremstyle{plain}
\newtheorem{thm}{Theorem}[section]
\newtheorem*{thma}{Theorem A}
\newtheorem*{thmb}{Theorem B}
\newtheorem*{thmc}{Theorem C}
\newtheorem*{thmd}{Theorem D}
\newtheorem{Prop}[thm]{Proposition}
\newtheorem{Def}[thm]{Definition}
\newtheorem{lma}[thm]{Lemma}
\newtheorem{cor}[thm]{Corollary}

\newtheorem{problem}{Problem}
\newtheorem{remark}[thm]{Remark}
\newtheorem{ex}{Example}

\newcommand{\Dsp}{Dsp}
\newcommand{\Dsl}{Dsl}
\newcommand{\supp}{supp}
\newcommand{\GP}{GP}
\newcommand{\Ent}{\mathrm{Ent}}

\newcommand{\Ch}{\mathrm{Ch}}
\newcommand{\AC}{\mathrm{AC}}
\newcommand{\Ric}{\mathrm{Ric}}

\newcommand{\Dom}{\mathit{Dom}}
\newcommand{\Lip}{\mathrm{Lip}}
\newcommand{\F}{\mathcal{F}}
\newcommand{\E}{\mathcal{E}}
\newcommand{\lip}{\mathrm{lip}}

\title{Gradient flow for the Boltzmann entropy and Cheeger's energy on time-dependent metric measure spaces}
\author{Eva Kopfer\thanks{Institut für Angewandte Mathematik, Universität Bonn, Endenicher Allee 60, 53115 Bonn, Germany (\texttt{eva.kopfer@iam.uni-bonn.de})}}
\date{}

\begin{document}

\maketitle

\begin{abstract}
We introduce notions of dynamic gradient flows on time-dependent metric spaces as well as on time-dependent Hilbert spaces.
We prove existence of solutions for a class of time-dependent energy functionals in both settings.
In particular in the case when each underlying space satisfies a lower Ricci curvature bound in the sense of Lott, Sturm and Villani, we provide time-discrete approximations of the time-dependent heat flows introduced in \cite{sturm2016}. 
\end{abstract}

\tableofcontents

\section{Introduction}
The heat flow in $\mathbb{R}^n$
can be  interpreted as the gradient flow of the Dirichlet energy 
$ \textrm{Dir}(u)=\frac12\int_{\mathbb{R}^n}|\nabla u|^2dx$
in the Hilbert space $L^2(\mathbb{R}^n)$. 
Since the seminal work of Jordan, Kinderlehrer and Otto in \cite{otto} it is known that the heat flow equivalently arises as the gradient flow of the entropy functional
$ \Ent(\rho\mathcal{L}^n)=\int_{\mathbb{R}^n}\rho\log\rho dx$
in the space of Borel probability measures with finite second moments endowed with the $L^2$-Kantorovich distance $W_2$. This has been accomplished by using a discrete variational scheme.
By now similar interpretations of the heat flow have been established in Riemannian manifolds \cite{erbar}, in Finsler spaces \cite{ohta/sturm},
in Alexandrov spaces \cite{gko} and in CD$(K,\infty)$-spaces \cite{gigli2009}. 

The class of CD$(K,\infty)$-spaces has been independently introduced by Lott and Villani in \cite{ls} and Sturm in \cite{sturm2006} and consists of metric measure spaces such that the relative entropy is $K$-convex along geodesics in the $L^2$-Kantorovich space. For a Riemannian manifold this is equivalent to saying that its Ricci curvature is bounded from below by $K$. Lower Ricci curvature bounds are intimately linked to the heat equation in the sense that they can be characterized by contraction estimates of the heat flow, see e.g. \cite{sturmrenesse, eks2014}.

A similar result can be shown for time-dependent manifolds which evolve under a super-Ricci flow, see e.g. \cite{mccanntopping, HN2015}. In \cite{sturm2015} Sturm introduced a notion for super-Ricci flows on metric measure spaces in terms of `dynamic convexity' of the relative entropy. The relative entropy is now a time-dependent functional from the time-dependent $L^2$-Kantorovich space.  

In this paper we present notions of gradient flows for time-dependent functionals from time-dependent metric spaces as well as time-dependent Hilbert spaces. We will prove existence for functionals which are uniformly $K$-convex using an adapted discrete variational scheme. The relative entropy and Cheeger's energy will serve as a role model for this.

In the static CD$(K,\infty)$ setting the heat flow can be unambiguously defined as the gradient flow of the relative entropy or the gradient flow of the Cheeger's functional. In the time-dependent setting the picture is less complete. However in some sense we identify the gradient flow of the relative entropy with the gradient flow of Cheeger's energy via the heat flows in \cite{sturm2016} provided that each underlying space is RCD$(K,\infty)$.  The condition RCD$(K,\infty)$ stands for CD$(K,\infty)$ combined with infinitesimally Hilbertian, where the latter means that Cheeger's energy is a bilinear form. The bilinearity of Cheeger's energy together with some regularity assumptions allow the authors in \cite{sturm2016} to prove existence and uniqueness of the heat flow in time-dependent metric measure spaces via the general theory of coercive operators. Here we identify the forward adjoint heat flow from \cite{sturm2016} with the gradient flow of the entropy provided that each space is RCD$(K,\infty)$, and the heat flow from \cite{sturm2016} with the gradient flow of Cheeger's energy. In particular we provide an explicit construction of the trajectory of each heat flow. 

In the following we will briefly present our main results for our model functionals.

Throughout this paper let $X$ be a topological space equipped with
\begin{itemize}
 \item a one-parameter family $d_t$ of complete geodesic separable metrics,
 \item and a one-parameter family of Borel measures such that $m_t=e^{-f_t}m$ for some probability measure $m$ and suitable functions $f_t$.
\end{itemize}
Further we assume the following
\begin{align}\label{introeq}
 |\log( d_t(x,y)/d_s(x,y))|\leq L|t-s|,\quad\text{ and }\quad|f_t(x)-f_s(x)|\leq L^*|t-s|.
\end{align}
\subsubsection*{Dynamic gradient flow of the entropy}
Given two probability measures $\mu,\nu$ with finite second moments (with respect to any metric $d_t$ and denoted by $\mathcal P_2(X)$) we define for every $t\in[0,T]$ the $L^2$-Kantorovich distance by
\begin{equation*}
 W_t(\mu,\nu)=\inf\left\{\int_{X\times X}d_t^2(x,y)\,d\pi(x,y)\Big|\pi\text{ is a coupling of }\mu\text{ and }\nu\right\}^{1/2}.
\end{equation*}
The relative entropy $S_t$ on $\mathcal P_2(X)$ is defined by
\begin{align*}
 S_t(\mu):=\int_X\rho\log\rho \,d m_t,
\end{align*}
provided that $\mu$ has a density $\rho$ with respect to $m_t$. We assume that each static
space $(X,d_t,m_t)$ has Ricci curvature bounded below by some $K\in\mathbb R$, i.e. for each $t$ and each $\mu,\nu$ there exists a $W_t$-geodesic 
$(\rho_a)_{a\in[0,1]}$
connecting $\mu$ and $\nu$ such that
\begin{align}\label{introconvex}
 S_t(\rho_a)\leq (1-a)S_t(\mu)+aS_t(\nu)-\frac{Ka(1-a)}2W_t^2(\mu,\nu).
\end{align}
This "uniform" geodesic convexity holds particularly true when the underlying spaces evolve as a super-Ricci flow introduced in \cite{sturm2015} provided that \eqref{introeq} holds.

We are interested in defining a notion of gradient flows for the time-dependent entropy functional on the space of probability measures over $X$. 
For this we adapt the discrete variational scheme to our dynamic setting, which we describe in the following.
Fix a time step $h>0$
and an initial probability measure $\mu_0$. Recursively define for every $n\in \mathbb N$ such that $nh\leq T$ the minimizer $\mu_n^h$ by
\begin{align}\label{dvs}
\mu_0^h:=\mu_0,\qquad \mu_n^h:=\arg\min_\nu\left(S_{nh}(\nu)+\frac1{2h}W^2_{nh}(\mu_{n-1}^h,\nu)\right).
\end{align}
We then define a discrete trajectory as the piecewise constant interpolant $(\bar\mu_t^h)_{t\in[0,T]}$ by 
\begin{align}\label{inter}
 \bar\mu_0^h:=\mu_0,\qquad \bar\mu_t^h:=\mu_n^h \quad\text{ if }t\in((n-1)h,nh].
\end{align}
By the direct method of the calculus of variations one can prove that there exists a unique solution to \eqref{dvs}. Having established a sequence of minimizers
$(\mu_n^h)_n$ the next step is to prove the existence of a curve $(\mu_t)_{t\in[0,T]}$ such that $\bar\mu_t^h\to\mu_t$ weakly as $h\to0$. 
From the convergence of the interpolants $\bar\mu_t^h$ to some limit curve $\mu_t$ we can deduce that $\mu_t$ satisfies a 
\emph{dynamic energy dissipation inequality}, in short dynamic EDI (cf. Theorem \ref{existence}). Moreover this curve is unique (cf. Theorem \ref{uniqueness}).

\begin{thma}\label{introheat1}
Suppose that each $(X,d_t,m_t)$ satisfies a CD$(K,\infty)$ condition. Then there exists an absolutely continuous curve 
$(\mu_t)_{t\in[0,T]}\subset\mathcal P_2(X)$ and a subsequence $h_n\to0$ as $n\to\infty$ such that 
 \begin{align*}
  \bar \mu_t^{h_n}\to\mu_t, \text{ for every }t\in[0,T] \text{ as }n\to\infty,
 \end{align*}
where the convergence is to be understood in duality with bounded continuous function on $X$. Moreover this curve satisfies
 \begin{align}\label{intede}
   S_t(\mu_t)+\frac12\int_0^t|\dot\mu_r|^2_rdr+\frac12\int_0^t|\nabla_rS_r|^2(\mu_r)dr
  =S_0(\mu_0)+\int_0^t(\partial_r S_r)(\mu_r)dr.
 \end{align}                                            

\end{thma}

\subsubsection*{Dynamic gradient flow of Cheeger's energy}
For each $t\in[0,T]$ let us denote by $\Ch_t\colon L^2(X,m)\to[0,\infty]$ Cheeger's energy
\begin{equation*}
 \Ch_t(u)=\frac12\inf\left\{\liminf_{n\to\infty}\int_X(\lip_tu_n)^2dm_t \Big|u_n\in \Lip(X),\int_X|u_n-u|^2dm_t\to0\right\},
\end{equation*}
where $\lip_tu$ denotes the local slope.
By making use of the minimal relaxed gradient $|\nabla_tu|_*$ (\cite[Definition 4.2]{agscalc}),
this functional admits an integral representation
\begin{equation*}
 \Ch_t(u)=\frac12\int_X|\nabla_tu|_*^2dm_t,
\end{equation*}
set equal to $+\infty$ if $u$ has no relaxed slope.
The subdifferential $D^-_t \Ch_t(u)$ of $\Ch_t$ at some $u\in \Dom(\Ch_t)$ is the set of $v\in L^2(X,m)$ such that 
\begin{equation*}
 \Ch_t(w)-\Ch_t(u)\geq \langle v,w-u\rangle_t\qquad \forall w\in L^2(X,m),
\end{equation*}
where we set $\langle v,w\rangle_t=\int vw\,dm_t$.
 We prove the existence of a dynamic gradient flow for $(\Ch_t)$, cf. Theorem
\ref{exi}, via a discrete variational scheme similar to \eqref{dvs}. Since each $\Ch_t$ already defines a convex functional, we do not need any curvature condition on the space.
\begin{thmc}\label{introheat3}
 Let $\bar u\in \Dom(\Ch)$. Then there exists a unique gradient flow for $(\Ch_t)_{t\in[0,T]}$ starting in $\bar u$, 
 i.e. an absolutely continuous curve $u\colon[0,T]\to \Dom(\Ch)$ solving
 \begin{align}\label{eq:introheat3}
  \partial_t u_t\in-D^-_t\Ch_t(u_t)\quad \text{ for a.e. }t\in(0,T)
 \end{align}
 and $\lim_{t\to0}u_t=\bar u$.
 \end{thmc}
 
 Let us emphasize here that many facts from the static theory of gradient flows hold no longer true in the time-dependent framework. For example it is no longer valid that we have convergence of minimizers (Example \ref{convofmin}), or `EVI' implies `EDE' (Example \ref{eviimpliesnoedi}), or a `minimal selection principle' (Example \ref{minsel}).

 \subsubsection*{Heat flows on time-dependent metric measure spaces}
In the static CD$(K,\infty)$ setting it is a well-known fact that the heat flow can be unambiguously defined as the gradient flow of the 
entropy or as the gradient flow of Cheeger's energy. Here we provide a similar result using the heat and the adjoint heat flow introduced in \cite{sturm2016}. 

Let us assume that each $(X,d_t,m_t)$ is infinitesimally Hilbertian so that $\E_t:=2\Ch_t$ becomes a Dirichlet form with generator $\Delta_t$. The heat flow solving $\partial_tu=\Delta_tu$ as well as the adjoint heat flow solving
$\partial_sv_s=-\Delta_sv_s+\dot f_sv_s$ has to be understood in a weak distributional sense.  If we assume further that $|f_t(x)-f_t(y)|\leq Cd_t(x,y)$ the existence and uniqueness of both flows are ensured by the general theory of time-dependent coercive operators on some fixed Hilbert space $L^2(X,m_{t_0})$. The flows denoted by $P_{t,s}u$ and $P_{t,s}^*v$ are adjoint in the sense that $\int P_{t,s}u v\, dm_t=\int u P_{t,s}^*v\, dm_s$. Let us remark that many properties which are apparent for the static heat flow on metric measure spaces hold no longer true for the time-dependent version, or require some extra effort, e.g. semigroup and generator commute or the semigroup maps $L^2$ into the domain of the generator.

The following theorem states that the trajectory of the adjoint heat flow parametrized forwards in time coincides with the trajectory of the gradient flow of the relative entropy.
For the precise statement see Theorem \ref{identico}.
\begin{thmb}\label{introheat2}
Suppose that each $(X,d_t,m_t)$ satisfies an RCD$(K,\infty)$ condition.
 Let $(\mu_t)_{t\in[0,T]}$ be a continuous curve in $\mathcal P_2(X)$. Then the following are equivalent:
 \begin{enumerate}
  \item $(\mu_t)$ is a gradient flow for the relative entropy,
  \item $(\mu_t)$ is given by $\mu_t(d x)=\rho_t(x)m_t(d x)$, where $(\rho_t)$ is a solution to the adjoint heat equation
  \begin{align*}
   \partial_t\rho_t(x)=\Delta_t\rho_t(x)+\rho_t(x)\partial_tf_t(x).
  \end{align*}

 \end{enumerate}

\end{thmb}

In a second step we identify the gradient flow with the heat flow, see Theorem \ref{idi}. Hence under the RCD$(K,\infty)$ assumption the gradient flow of the relative entropy and the gradient flow of Cheeger's energy
are connected by the heat and the adjoint heat flow.
 \begin{thmd}\label{introheat4}
Suppose that each $(X,d_t,m_t)$ is infinitesimally Hilbertian.
Let $(u_t)$ be a continuous curve in $L^2(X,m)$. Then the following are equivalent:
\begin{enumerate}
\item $(u_t)$ is a gradient flow for Cheeger's energy,
\item $(u_t)$ is a solution to the heat equation
$$\partial_t u_t=\Delta_t u_t.$$
\end{enumerate}
\end{thmd}

\subsubsection*{Related work}
Gradient flow formulations for time-dependent functionals similar to \eqref{intede} and \eqref{eq:introheat3} have been considered recently.
In \cite{rms}, Rossi, Mielke and Savar\'{e} analyze doubly nonlinear evolution equations on a reflexive Banach space $V$
\begin{equation*}
 D^-\psi(\partial_t u_t)+D^- E_t(u_t)\ni 0 \, \text{ in }V^*\, \text{ for a.e. }t\in(0,T).
\end{equation*}
They propose a
formulation of a similar form as \eqref{intede} and prove existence of solutions by making use of a discrete variational scheme. 
In \cite{ferreira}, Ferreira and Valencia-Guevara consider the Fokker-Planck equation
\begin{align*}
 \partial_t\rho=\kappa(t)\Delta\rho+\nabla\cdot(\nabla V(t,x)\rho) \text{ on }\mathbb R^d\times [0,\infty),
\end{align*}
for some fixed non-increasing absolutely continuous function $\kappa\colon[0,\infty)\to[0,\infty)$, and potential
$V\colon[0,\infty)\times\mathbb R^d\to\mathbb R$. They show that the solution can be obtained as the gradient flow of some specified 
time-dependent energy functional. They also use a version of the discrete variational scheme similar to \cite{rms} to obtain existence. In both papers the authors do not
treat time-dependent metric measure spaces as done here.

\subsubsection*{Organization of the article} 
In Section \ref{secgrad} we briefly recall the concept of gradient flows in metric spaces.
In Section \ref{sectime} we introduce the notion of dynamic EDI- and dynamic EDE-gradient flows on time-dependent metric spaces $(X,d_t)_{t\in[0,T]}$
satisfying \eqref{introeq}. Moreover we introduce the notion of dynamic EVI-gradient flows and show that it implies EDE restricted to a suitable class of energy functionals.  
We show existence of dynamic EDI-gradient flows for a class of energy functionals 
$E\colon [0,T]\times X\to (-\infty,+\infty]$ and give sufficient conditions for the existence of EDE-gradient flows.
In Section \ref{secent} we apply the results from Section \ref{sectime} and prove existence and uniqueness of dynamic EDI-gradient flows in
time-dependent metric measure spaces $(X,d_t,m_t)_{t\in[0,T]}$ for the
time-dependent entropy functional $S\colon[0,T]\times \mathcal P_2(X)\to(-\infty,+\infty]$. In Section \ref{sechilbert} we consider
dynamic gradient flows in the form of \eqref{eq:introheat3} on time-dependent Hilbert spaces
$(H,\langle\cdot,\cdot\rangle_t)_{t\in[0,T]}$ and show that they imply EVI.
We prove existence and uniqueness of such gradient flows for a class of energy functionals 
$E\colon[0,T]\times H\to [0,+\infty]$. In Section \ref{secheat} we recall the concept of heat equation on time-dependent metric measure
spaces introduced in \cite{sturm2016}. We identify the dynamic EDI-gradient flow of the entropy with the forward adjoint heat flow. 
We apply the results from Section \ref{sechilbert}
and directly obtain existence and uniqueness of a dynamic gradient flow for Cheeger's energy and identify it with the heat flow.

\subsubsection*{Acknowledgement}
I would like to thank my supervisor Karl-Theodor Sturm for supporting me. I also thank Peter Gladbach for helpful comments and suggestions on this paper.
\section{Gradient flows in metric spaces}\label{secgrad}
We briefly recall  the notions of gradient flows on metric spaces $(X,d)$. 
A curve $x\colon[a,b]\to X$ is said to belong to $\AC^p([a,b];X)$ for $1\leq p\leq \infty$, if there exists $g\in L^p(a,b)$ such that
\begin{align}\label{eq:AC}
 d(x_s,x_t)\leq\int_s^tg(r)dr\qquad\text{ for every }a\leq s\leq t\leq b.
\end{align}
The \emph{metric speed} of $x$, defined by
\begin{align*}
 |\dot x_t|:=\lim_{h\to0}\frac{d(x_{t+h},x_{t})}{h},
\end{align*}
exists for a.e. $t\in(a,b)$, is of class $L^p(a,b)$ and is the smallest function such that \eqref{eq:AC} holds, 
see e.g. \cite[Theorem 1.1.2]{ag}.

Given $E\colon X\to (-\infty,+\infty]$ we define the \emph{slope} $|\nabla E|(x)$ at $x$ by
\begin{align*}
 |\nabla E|(x):=\limsup_{y\to x}\frac{(E(x)-E(y))^+}{d(x,y)}.
\end{align*}
We now are ready to give three possible definitions of gradient flows in a metric framework, cf. \cite{ags}.
\begin{Def}
\begin{enumerate}
 \item An absolutely continuous curve $(x_t)\subset X$ is a EDI-gradient flow if it satisfies the following \textit{Energy Dissipation Inequality}
\begin{equation}\label{EDI}
 E(x_s)+\frac12\int_t^s|\dot x_r|^2dr+\frac12\int_t^s|\nabla E|^2(x_r)dr\leq E(x_t)\quad \forall s<t.
\end{equation}
\item An absolutely continuous curve $(x_t)\subset X$ is a EDE-gradient flow if it satisfies the following \textit{Energy Dissipation Equality}
\begin{equation}\label{EDE}
 E(x_s)+\frac12\int_t^s|\dot x_r|^2dr+\frac12\int_t^s|\nabla E|^2(x_r)dr= E(x_t)\quad \forall s<t.
 \end{equation}
 \item An absolutely continuous curve $(x_t)\subset X$ is a EVI-gradient flow (wit respect to $\lambda\in\mathbb R$) if it satisfies the following 
 \textit{Evolution Variation Inequality}
\begin{equation}\label{EVI}
 E(x_t)+\frac12\partial_td^2(x_t,y)+\frac{\lambda}2d^2(x_t,y)\leq E(y) \quad \text{ for a.e. }t\in[0,T],\forall y\in X.
 \end{equation}
\end{enumerate}
\end{Def}
It holds in general that EVI implies EDE, and EDE implies trivially EDI.
If the underlying space is a Hilbert space and the energy functional is convex, all three formulations are equivalent. Moreover we can characterize
the flow in terms of the subdifferential by
\begin{align}\label{SDF}
 \dot x_t\in -D^-E(x_t),
\end{align}
where $D^-E(x)$ consists of all $v\in X$ such that
$$E(x)+\langle v,y-x\rangle \leq E(y)\quad \forall y\in X.$$

In this paper we are interested in finding substitutions for formulations of the form \eqref{EDI} and \eqref{EDE}, where the metric as well as the functional varies in time.  
A formulation in the sense of \eqref{EVI} has already been introduced in \cite{sturm2016}. Moreover, in the Hilbert space case, we study the time-dependent counterpart of relations of the form \eqref{SDF}.

\section{Dynamic gradient flows in time-dependent metric spaces}\label{sectime}
In the sequel we fix a one-parameter family of complete geodesic metric spaces $(X,d_t)_t$ indexed by $t\in[0,T]$.  
We always assume that the map $t\to \log d_t(x,y)$ is Lipschitz continuous, i.e. there exists a constant $L$
such that
\begin{equation}\label{loglip}
 |\log( d_t(x,y)/d_s(x,y))|\leq L|t-s|.
\end{equation}
We give a simple example for this setting. 
\begin{ex}
Let $M$ be a smooth closed manifold equipped with a smooth family of Riemannian metrics $(g_t)$
evolving under a Ricci flow, i.e.
\begin{align*}
\frac12 \partial_tg_t=-\Ric(g_t),
\end{align*}
where $\Ric(g)$ denotes the Ricci curvature. At least for short time intervals we have existence and uniqueness of such a flow (see e.g. Theorem 5.2.1 in 
\cite{topping2006lectures}). Under the assumption that the curvature does not blow up ($|\Ric|\leq L$), we have metric equivalence 
\begin{align*}
 |\partial_t\log g_t(v,v)|\leq L.
\end{align*}
This implies that \eqref{loglip} holds for
the geodesic distances $(d_t)$. 
\end{ex}

\subsubsection*{The metric speed}
\begin{Def}\label{abscont}
Let $[0,T]\ni t\mapsto x_t\in X$ be a curve. We say that $(x_t)\in \AC^p([0,T];X)$, for $p\in[1,\infty]$, if for any 
(and thus for all) $t^*\in[0,T]$ there exists a function $g\in L^p(0,T)$ such that 
\begin{equation*}\label{eq:abscont}
 d_{t^*}(x_t,x_s)\leq \int_t^sg(r)dr \quad \forall 0\leq t\leq s\leq T.
\end{equation*}

\end{Def}
We define 
the length of a curve $x\colon[0,T]\to X$ to be
\begin{align*}
 L_x(t)=\lim_{h\to0}\sup\left\{\sum_{i=1}^nd_{t_j}(x_{t_j},x_{t_{j+1}}): 0=t_1<\ldots <t_n=t, t_{j+1}-t_j\leq h\right\}.
\end{align*}
It is a direct consequence of the definition of $L_x(t)$ that if $x_n\to x$ pointwise as $n\to\infty$ we have 
$L_{x}(t)\leq \liminf_{n\to\infty}L_{x_n}(t)$ for every $t\in[0,T]$.

Note that $L_x$ is absolutely continuous as soon as $x$ is and hence we may define the \emph{instantaneous speed} of the curve as the derivative of its length.
\begin{equation*}
 |\dot x|_t:=\dot L_x(t).
\end{equation*}

\begin{lma}\label{metricspeed}
 For any curve $x\in \AC^p([0,T];X)$ the 
  function $t\mapsto |\dot x|_t$ is in $L^p(0,T)$, and for almost every $t\in(0,T)$
  \begin{equation*}
   |\dot x|_t=\lim_{s\to t}\frac{d_t(x_s,x_t)}{|s-t|}=|\dot x_t|_t.
  \end{equation*}

\end{lma}
\begin{proof} 
If we show the second assertion the first assertion is an easy consequence of \eqref{loglip}. Let $x$ be an absolutely continuous curve and choose
an arbitrary partition $s=t_1<t_2<\ldots <t_{N+1}=t$. Then we find
\begin{align*}
 d_t(x_t,x_s)\leq \sum_{i=1}^Nd_{t_i}(x_{t_i},x_{t_{i+1}})+C\sum_{i=1}^N|t-t_i|d_{t}(x_{t_i},x_{t_{i+1}})\\
 \leq \sum_{i=1}^Nd_{t_i}(x_{t_i},x_{t_{i+1}})+C|t-s|\int_s^tg(r)dr,
\end{align*}
where we used \eqref{loglip} and $g$ is some $L^p$ function. Hence we may estimate
\begin{equation*}
  d_t(x_t,x_s)\leq L(t)-L(s)+C|t-s|\int_s^tg(r)dr.
\end{equation*}
Dividing by $|t-s|$ and letting $s\to t$ we deduce
\begin{equation*}
 \limsup_{s\to t}\frac{d_t(x_t,x_s)}{|t-s|}\leq \dot L_x(t) \text{ for almost every }t.
\end{equation*}
We show the other inequality by contradiction. Fix $\eta>0$ and consider the set of points
\begin{equation*}
 F=\left\{t:\liminf_{s\to t}\left(\frac{d_t(x_s,x_t)}{|s-t|}-\frac1{|s-t|}\int_s^t\dot L_x(r)dr\right)<-\eta\right\}.
\end{equation*}
We assume that the Lebesgue outer measure $\mathcal{L}^*(F)>0$. Fix $\delta>0$ and cover the set $F$ with intervals
\begin{align*}
 \mathcal{F}:=\bigcup_{t\in F}(t-\delta_t,t+\delta_t), \text{ where } \delta_t<\delta,
\end{align*}
such that 
\begin{align}\label{violated}
 d_t(x_t,x_s)<\int_s^t\dot L_x(r)dr-|t-s|\eta/2
\end{align}
for all $t\in F$ and some $s\in(t-\delta_t,t+\delta_t)$. From the Besicovitch covering theorem \cite[Theorem 5.8.1]{bogachev2007}
it follows that there exists a constant $N$ and a subcollection $\mathcal{F}_1,\cdots,\mathcal{F}_N$ each consisting of at most countably many disjoint intervals $B$ such that
\begin{equation*}
 F\subset \bigcup_{i=1}^N\bigcup_{B\in\mathcal{F}_i}B. 
\end{equation*}
Since the outer measure of $F$ is strictly positive we can find a family $\mathcal{F}_j$ of at most countably many disjoint intervals denoted by 
 $\mathcal{F}_j=\{(t_i-\delta_i,t_i+\delta_i), i\in I\}$
such that $\mathcal{L}^1(\bigcup_{B\in \mathcal{F}_j} B)\geq\frac1N\mathcal{L}^*(F)>0$.

We define a curve $x^\delta\colon[0,T]\to X$ in the following way
\begin{equation*}
   x^\delta_t =
   \begin{cases}
     x_{t_i} & \text{if } t \in (t_i,t_i+\delta_i) \\
     x_{t}    &\text{else.}
   \end{cases}
\end{equation*}
Note that this curve is not continuous but still its length is finite.
 Further we observe that $x_t^\delta$ converges to $x_t$ pointwise as $\delta$ goes to 0 and hence
\begin{equation*}
 \liminf_{\delta\to0}L_{x^\delta}(T)\geq L_x(T).
\end{equation*}
It suffices to show that 
\begin{equation}\label{deltalength}
 L_{x^\delta}(T)\leq L_x(T)(1+L\delta)-\frac\eta2\mathcal{L}^1\left(\bigcup_{i\in I}(t_i,t_i+\delta_i)\right),
\end{equation}
since then 
\begin{align*}
L_x(T)\leq \liminf_{\delta\to0}L_{x^\delta}(T)\leq L_x(T)-\frac\eta{4N}\mathcal{L}^*(F)<L_x(T),
\end{align*}
which is clearly a contradiction. Hence for the outer measure it must hold $\mathcal{L}^*(F)=0$ and therefore already $\mathcal{L}^1(F)=0$. 
Since $L_x$ is absolutely continuous we conclude 
\begin{equation*}
 \liminf_{s\to t}\frac{d_t(x_s,x_t)}{|s-t|}\geq \dot L_x(t) \text{ for almost every }t\in[0,T].
\end{equation*}
It remains to show \eqref{deltalength}. Take a partition $(p_j)_{j=1}^m$ of $[0,T]$, with $0<p_{j+1}-p_j\leq h$ and $h<<\delta$.
Consider the points near the boundary 
of $(t_i,t_i+\delta_i)$
\begin{equation*}
 j^\leq_i:=\max\{j|p_j\leq t_i, 1\leq j\leq m\}, \quad j^\geq_i =\min\{j|p_{j+1}\geq t_i+\delta_i, 1\leq j\leq m \}.
\end{equation*}
Since $x$ is absolutely continuous we can estimate
\begin{equation*}
 d_{p_{j^\leq_i}}(x^\delta_{p_{j^\leq_i}},x^\delta_{p_{j^\leq_i+1}})=d_{p_{j^\leq_i}}(x_{p_{j^\leq_i}},x_{t_i})\leq \int_{p_{j^\leq_i}}^{p_{j^\leq_i}+h}g(r)dr,
\end{equation*}
where $g\in L^p(0,T)$. Applying \eqref{violated}, \eqref{loglip} and again the absolute continuity we obtain
\begin{align*}
 &d_{p_{j^\geq_i}}(x^\delta_{p_{j^\geq_i}},x^\delta_{p_{j^\geq_i+1}})=d_{p_{j^\geq_i}}(x_{t_i},x_{p_{j^\geq_i+1}})\\
 &\leq d_{p_{j^\geq_i}}(x_{t_i},x_{t_i+\delta_i}) + d_{p_{j^\geq_i}}(x_{t_i+\delta_i},x_{p_{j^\geq_i+1}})\\
 &\leq d_{t_i}(x_{t_i},x_{t_i+\delta_i})(1+L\delta_i)+\int_{t_i+\delta_i}^{t_i+\delta_i+h}g(r)dr\\
 &\leq \int_{t_i}^{t_i+\delta_i}(\dot L_x(r)-\eta/2)dr(1+L\delta_i)+\int_{t_i+\delta_i}^{t_i+\delta_i+h}g(r)dr.
\end{align*}
Taking the supremum over all partitions $(p_j)$ and letting $h\to0$ we can estimate the length of the curve $x^\delta$
\begin{align*}
 L_{x^\delta}(T)&\leq \int_{(0,T)\setminus\bigcup_i(t_i,t_i+\delta_i)}\dot L_x(r)dr+\sum_i\int_{t_i}^{t_i+\delta_i}(\dot L_x(r)-\eta/2)dr(1+C^*\delta_i)\\
 &\leq\int_{(0,T)}\dot L_x(r)dr(1+L\delta)-\eta/2\mathcal{L}^1\left(\bigcup_i(t_i,t_i+\delta_i)\right),
\end{align*}
which proves \eqref{deltalength}.

\end{proof}

\subsubsection*{The slope}

\begin{Def}
Let $E\colon [0,T]\times X\to (-\infty,+\infty]$
and $s,t\in[0,T]$, $x\in X$ with $E_t(x)<\infty$. Then the \emph{slope}
 $|\nabla_sE_t|(x)$ of $E_t$ with respect to $d_s$  is given by
 \begin{equation*}
  |\nabla_s E_t|(x)=\limsup_{y\to x}\frac{[E_t(x)-E_t(y)]^+}{d_s(x,y)}=\limsup_{y\to x}\max\left\{\frac{E_t(x)-E_t(y)}{d_s(x,y)},0\right\}.
 \end{equation*}

\end{Def}
We mainly deal with the case $t=s$ in the definition of the slope. 
We estimate the deviation of the $d_t$ slope from the $d_s$ slope in the 
following lemma.
\begin{lma}\label{modest}
 Let $s,t\in[0,T]$ and $x\in X$ such that $E_t(x)<\infty$ and $|\nabla_sE_t|(x)<\infty$. Then $|\nabla_tE_t|(x)<\infty$ and
 \begin{equation*}
  ||\nabla_tE_t|(x)-|\nabla_sE_t|(x)|\leq L|t-s||\nabla_sE_t|(x).
 \end{equation*}

\end{lma}
\begin{proof}
 This follows from \eqref{loglip} and $\log r\leq r-1$ and $\log(r^{-1})\geq 1-r$ respectively.
\end{proof}

\subsection{Dynamic EDI- and EDE-gradient flows}
Let us first motivate the definition of dynamic EDI-gradient flows by
considering a Hilbert space $X$ endowed with a family of scalar products $(\langle\cdot,\cdot\rangle_t)$ depending smoothly on $t$. Let $E_t\colon X\to \mathbb R$
be a $\mathcal C^1$ functional also smoothly depending on time. In this setting we understand a gradient flow as a curve solving
\begin{equation}\label{eq:dyngradientflow}
 \dot x_t=-\nabla_t E_t(x_t).
\end{equation}
Let us observe that \eqref{eq:dyngradientflow} can be rewritten as
\begin{equation}\label{smoothedi}
 \frac{d}{dt}E_t(x_t)\leq -\frac12|\nabla_tE_t|_t^2(x_t)-\frac12|\dot x_t|_t^2+(\partial_tE_t)(x_t),
\end{equation}
where $(\partial_tE_t)(x_t)$ stands for $\frac{d}{ds}E_s(x_t)\Big|_{s=t}$.
Indeed, along any differentiable curve it holds
\begin{align*}
 \frac{d}{dt} E_t(x_t)=\frac{d}{ds}E_s(x_t)\Big|_{s=t}+\langle\nabla_t E_t(x_t),\dot x_t\rangle_t\\
 \geq \frac{d}{ds}E_s(x_t)\Big|_{s=t}-\frac12|\nabla_t E_t|_t^2(x_t)-\frac12|\dot x_t|_t^2,
\end{align*}
and we have equality if and only if \eqref{eq:dyngradientflow} holds. The functional's dependence on the time variable leads to a ``drift'' of the 
gradient flow, i.e. in some sense the gradient flow does not follow the steepest descent. To illustrate this we give an example about the asymptotic behavior.
\begin{ex}\label{convofmin}
 Let $X=\mathbb R$ and $d_t(x,y)=|x-y|$ for $t\in[0,\infty)$ and $x,y\in\mathbb R$. We consider the energy $E_t(x)=(x-t)^2$ and the curve $x_t=\frac12e^{-2t}+t-\frac12$.
 Note that 
 \begin{align*}
  \dot x_t=-e^{-2t}+1=-2(x_t-t)=-\partial_x E_t(x),
 \end{align*}
 and hence $(x_t)$ is a gradient flow. A well-known fact in the theory of gradient flows 
 is that for strictly convex functionals the gradient flow converges to the 
 minimum of the functional as $t\to\infty$, see e.g. \cite[Theorem 3.1(v)]{ag}. In our case the minima depend on time and are given by $x^{\min}_t=t$.
 Hence
 \begin{align*}
  |x_t-x_t^{\min}|=\frac{|e^{-2t}-1|}2,
 \end{align*}
which obviously does not converge to 0 as $t\to\infty$.

\end{ex}

Let us now come back to our original family of complete, separable, geodesic, metric spaces $(X,d_t)$ such that \eqref{loglip} holds true.
We call a measurable functional $E$ on $[0,T]\times X$ \emph{admissible} if it satisfies the following assumptions.

\begin{description}
  \item[A1] The domain $\Dom(E_t):=\{x\in X|E_t(x)<\infty\}$ is time-independent and nonempty.
  \item[A2] For each $t\in[0,T]$, $x\mapsto E_t(x)$ is uniformly bounded from below.
  \item[A3] For each $t\in[0,T]$, $x\mapsto E_t(x)$ is lower semicontinuous.
  \item[A4] The map $t\mapsto E_t(x)$ is uniformly Lipschitz continuous, i.e. there exists a constant $L^*$ such that 
  \begin{equation*}
   |E_t(x)-E_s(x)|\leq L^*|t-s| \quad\forall t,s\in[0,T], x\in \Dom(E),
  \end{equation*}
and the set of differentiability points of the map $t\mapsto E_t(x)$ can be chosen regardless 
  of $x\in X$ as soon as $x\in \Dom(E)$.

 \end{description}
 Note that the Lipschitz continuity of the map $t\mapsto E_t(x)$ provides a.e. differentiability in $t$ for every fixed $x$. But this is not enough to get
 a meaningful expression in \eqref{smoothedi}, since we may have that
 for some absolutely continuous curve $(x_t)$,
 $t\mapsto E_t$ is not differentiable at $x_t$ for every $t$. To circumvent this problem we suppose that the set of differentiability points can be chosen 
 independent of $x$, cf. \cite{ferreira,rms}. To illustrate this we give the following example, which has also been discussed in \cite{rms}.
 \begin{ex}\label{exnotdiff}
 Let $X=\mathbb R$ and $d_t(x,y)=|x-y|$ for every $t\in[0,T]$, $x,y\in\mathbb R$. Consider the following energy functional 
 $E\colon[0,T]\times \mathbb R\to[0,\infty)$ given by
 \begin{align*}
  E_t(x)=|x-t|.
 \end{align*}
Then the map $t\mapsto E_t(x)$ is clearly Lipschitz continuous with well-defined derivative $\partial_t E_t(x)$ as long as $t\in[0,T]\setminus\{x\}$.
If we choose the curve $(x_t)_{t\in[0,T]}\in\mathcal C^\infty([0,T])$ by setting $x_t=t$, the map $s\mapsto E_s(x_t)$ is not differentiable at any $t\in[0,T]$.
Indeed, for every $t\in[0,T]$ the right derivative $\partial_sE_s(x_t)_{|s=t+}$ equals $1$, while the left derivative $\partial_sE_s(x_t)_{|s=t-}$ equals $-1$.
 \end{ex}

\begin{Def}\label{defedi}
We call a locally absolutely continuous curve $x\colon[0,T]\to X$ a dynamic EDI-gradient flow for an admissible functional $E\colon [0,T]\times X\to (-\infty,\infty]$, 
if for every $t\in[0,T]$
 \begin{equation}\label{edi}
  E_t(x_t)+\frac12\int_0^t|\dot x|_r^2dr+\frac12\int_0^t|\nabla_rE_r|^2(x_r)dr\leq E_0(x_0)+\int_0^t(\partial_rE_r)(x_r)dr,
 \end{equation}
 where we used the shorthand notation $(\partial_tE_t)(x_t)=\frac{d}{dr}E_r(x_t)|_{r=t}$.
We call a locally absolute continuous curve $x\colon [0,T]\to X$ a dynamic EDE-gradient flow for an admissible functional $E\colon [0,T]\times X\to (-\infty,\infty]$, 
if for every $t\in[0,T]$
\begin{equation}\label{ede}
  E_t(x_t)+\frac12\int_0^t|\dot x|_r^2dr+\frac12\int_0^t|\nabla_rE_r|^2(x_r)dr= E_0(x_0)+\int_0^t(\partial_rE_r)(x_r)dr,
 \end{equation}
\end{Def}
Clearly, \eqref{ede} implies \eqref{edi}.
In the following we want to give sufficient conditions for the other implication.

\begin{Def}
 We say that the above mentioned functional $E$ is $K$-convex for $K\in \mathbb R$, if for every $t\in[0,T]$ and for any pair of points 
 $x,y\in \Dom(E)$ there exists a $d_t$-geodesic $(\gamma_a)_{a\in[0,1]}$ connecting $x$ and $y$ such that for all $a\in[0,1]$
 \begin{equation}\label{uniconvex}
  E_t(\gamma_a)\leq (1-a)E_t(\gamma_0)+aE_t(\gamma_1)-K\frac{a(1-a)}2d_t^2(\gamma_0,\gamma_1).
 \end{equation}

\end{Def}

The convexity assumption allows us to reformulate the slope
\begin{equation}\label{globalslope}
  |\nabla_tE_t|(x)=\sup_{y\neq x}\left[\frac{E_t(x)-E_t(y)}{d_t(x,y)}+\frac{K^-}2d_t(x,y)\right]^+,
 \end{equation}
with $K^-:=\max\{0,-K\}$,
cf. \cite[Theorem 2.4.9]{ags}.

The next proposition can be thought of as a weak chain rule in the sense of \cite[Proposition 3.19]{ag}. 
The convexity of the functional plays an important role in the proof
of this result.
Unlike in the static case we additionally have to impose a condition on the difference quotients of the functionals, cf. \cite[Theorem 5.4]{ferreira}.
\begin{Prop}\label{weakchainrule}
 Let $E\colon [0,T]\times X\to (-\infty,+\infty]$ be a $K$-convex admissible functional. Moreover assume that for almost every $t\in[0,T]$
 \begin{equation}\label{eq:weakchainrule0}
  \liminf_{n\to\infty}\frac{E_{t_n}(x_n)-E_t(x_n)}{t_n-t}\geq\partial_tE_t(x), \text{ if }t_n\searrow t,~ x_n\overset{d}\rightarrow x \text{ as }n\rightarrow\infty.
 \end{equation}
Then for every locally absolutely continuous curve $(x_t)\subset \Dom(E)$, the function $t\mapsto E_t(x_t)$ is absolutely continuous and it holds
 \begin{equation}\label{eq:weakchainrule}
  E_t(x_t)-E_s(x_s)\geq\int_s^t(\partial_rE_r)(x_r)\,dr-\int_s^t|\dot x|_r|\nabla_rE_r|(x_r)\,dr, \quad s<t.
 \end{equation}
In particular, if $(x_t)$ is a dynamic EDI-gradient flow, it is a dynamic EDE-gradient flow as well.
\end{Prop}
\begin{proof}
In view of \cite[Lemma 1.1.4(a)]{ags}
 we can find an increasing and absolutely continuous map $\boldsymbol{s}\colon[0,T]\to[0,T']$, whose inverse $\boldsymbol t$ is Lipschitz. The reparametrization 
 $\hat x_{\boldsymbol{s}(t)}:=x(t)$ 
 satisfies
 $|\dot{\hat{x}}_s|_{t^*}\leq 1$ for almost every $s\in[0,T']$ with respect to some fixed metric $d_{t^*}$. Notice that it is sufficient to prove 
 that $s\mapsto E_{\boldsymbol t(s)}(\hat x_s)=:\varphi(s)$ is absolutely continuous, as then $E_t(x_t)=E_{t}(\hat x_{\boldsymbol s(t)})$ is absolutely continuous and for
 almost every $t\in[0,T]$
 \begin{align*}
  \frac{d}{dt}E_t(x_t)&=\lim_{h\to0}\frac{E_{t+h}(x_{t+h})-E_t(x_t)}{h}\\
  &\geq \liminf_{h\to0}\frac{E_{t+h}(x_{t+h})-E_t( x_{t+h})}{h}
  +\liminf_{h\to0}\frac{E_{t}(x_{t+h})-E_t(x_t)}{d_t( x_{t+h}, x_t)}\frac{d_t(x_{t+h},x_t)}{h}\\
  &\geq \partial_tE_t( x_t)-\limsup_{h\to0}\frac{[E_{t}(x_{t})-E_t(x_{t+h})]^+}{d_t(x_{t+h},x_t)}\frac{d_t(x_{t+h},x_t)}{h}\\
  &\geq \partial_tE_t(x_t)-|\nabla_tE_t|(x_t)|\dot{ x}|_t,
  \end{align*}
  where we used \eqref{eq:weakchainrule0} in the third inequality. After integration we obtain  \eqref{eq:weakchainrule}.

In view of the convexity of $E$ we may use the representation formula of the slope \eqref{globalslope} and write using $a^+\leq (a+b)^++b^-$ and the Lipschitz 
property of the functional
\begin{equation}
 \begin{aligned}\label{onesidedlipschitz}
 \varphi(s_1)-\varphi(s_0)
  &\leq|\nabla_{\boldsymbol t(s_1)}E_{\boldsymbol t(s_1)}|(\hat x_{s_1})d_{\boldsymbol{t}(s_1)}(\hat x_{s_1},\hat x_{s_0})\\
  &+\frac {K^-}2d_{\boldsymbol{t}(s_1)}^2(\hat x_{s_1},\hat x_{s_0})+L^*|s_1-s_2|\\
 &\leq \Big(|\nabla_{\boldsymbol t(s_1)}E_{\boldsymbol t(s_1)}|(\hat x_{s_1})+\frac{K^-}2 D\Big)e^C|s_1-s_0|+L^*|s_1-s_0|,
\end{aligned}
\end{equation}
where $D$ is the finite diameter of the image $\{\hat x_s\}_s$ with respect to $d_{\boldsymbol{t}}$. Changing the roles of $s_0$ and $s_1$ yields 
\begin{align*}
 &| \varphi(s_1)-\varphi(s_0)|\\
 &\leq \Big(|\nabla_{\boldsymbol t(s_1)}E_{\boldsymbol t(s_1)}|(\hat x_{s_1})+|\nabla_{\boldsymbol t(s_0)}E_{\boldsymbol t(s_0)}|(\hat x_{s_0})+\frac{K^-}2 D\Big)e^C|t-s|+L^*|t-s|.
\end{align*}
Applying \cite[Lemma 1.2.6]{ags} we conclude that the map $s\mapsto \varphi(s)$ is in  the Sobolev space $W^{1,1}(0,T')$. To prove absolute continuity we simply 
check that it coincides 
with its continuous representative. We already know that $s\mapsto \varphi(s)$ is lower semicontinuous and therefore continuity follows if we show
\begin{equation*}
 \limsup_{\varepsilon\searrow0}\frac1{2\varepsilon}\int_{-\varepsilon}^\varepsilon \varphi(s+r)dr\leq \varphi(s)\quad\forall s\in(0,T').
\end{equation*}
This can be seen by applying \eqref{onesidedlipschitz} once more and we get
\begin{align*}
 &\limsup_{\varepsilon\searrow0}\frac1{2\varepsilon}\int_{-\varepsilon}^\varepsilon\varphi(s+r)-\varphi(s)\,dr\\
 \leq &\limsup_{\varepsilon\searrow0}\frac1{2\varepsilon}\int_{-\varepsilon}^\varepsilon \Big(|\nabla_{\boldsymbol t(s+r)}E_{\boldsymbol t(s+r)}|(\hat x_{s+r})+\frac{K^-}2 D\Big)e^C|r|+L^*|r|dr\\
 \leq  &\limsup_{\varepsilon\searrow0}\frac1{2}\int_{-\varepsilon}^\varepsilon \Big(|\nabla_{\boldsymbol t(s+r)}E_{\boldsymbol t(s+r)}|(\hat x_{s+r})+\frac{K^-}2 D\Big)e^C+L^*dr=0.
\end{align*}

\end{proof}

\subsection{Dynamic EVI$(K,\infty)$-gradient flows}\label{secevi}

Let us now come to the dynamic version of EVI$(K,\infty)$-gradient flows introduced in \cite{sturm2016}. 
\begin{Def}\label{ddist}
For $s,t\in[0,T]$ and an absolutely continuous curve $(x_a)_{a\in[0,1]}$, we define the action
 \begin{align*}
  \mathcal A_{s,t}(x)=\lim_{h\to0}\sup\Big\{\sum_{i=1}^n&(a_i-a_{i-1})^{-1}d_{\vartheta(a_{i-1})}^2(x_{a_{i-1}},x_{a_i})|\\
  &0=a_0<\dots < a_n=1,a_i-a_{i-1}\leq h\Big\},
 \end{align*}
 where $\vartheta\colon[0,1]\to[0,\infty)$ denotes the linear interpolation with $\vartheta(0)=s$ and $\vartheta(1)=t$.
 For two points $x^0,x^1\in X$ we define
 \begin{align*}
  d_{s,t}^2(x^0,x^1)=\inf\{\mathcal A_{s,t}(x)|x\colon[0,1]\to X \text{ absolutely continuous}, x_0=x^0,x_1=x^1\}.
 \end{align*}
\end{Def}
Note that
using the definition of the metric speed we obtain for the action the more intuitive expression, cf. \cite[Proposition 7.2]{sturm2016},
\begin{align*}
 \mathcal A_{s,t}(x)=\int_0^1|\dot x_a|^2_{\vartheta(a)}\, da.
\end{align*}
We understand $d_{s,t}(x,y)$ as ``dynamic distance'' between the points $x$ and $y$. In the next example we give a formula for $d_{s,t}$ in the special case $(\mathbb R^n,d_t)$, where $d_t^2(x,y)=\lambda t|x-y|^2$.
\begin{ex}
Consider $\mathbb R$ endowed with the distance $d_t^2(x,y)=\lambda t|x-y|^2$ for some $\lambda>0$. Then for each $x,y\in\mathbb R$ and $s<t$
\begin{align*}
d_{s,t}^2(x,y)=\frac{\lambda(t-s)}{\log t-\log s}|x-y|^2.
\end{align*}
Indeed, we have to compute the infimum of
$\int_0^1(s+a(t-s))(\dot\gamma_a)^2\, da$
among all $(\gamma_a)_{a\in[0,1]}$ such that $\gamma_0=x,\gamma_1=y$. From the Euler-Lagrange equation $0=\frac{d}{da}((s+a(t-s))\dot\gamma_a)$ we deduce that
\begin{align*}
\dot\gamma_a=\frac{t-s}{\log t-\log s}\frac{y-x}{s+a(t-s)}
\end{align*}
and hence 
\begin{align*}
\int_0^1(s+a(t-s))(\dot\gamma_a)^2\, da=\frac{t-s}{\log t-\log s}|y-x|^2,
\end{align*}
which proves the claim. The formula for $d_t$ on $\mathbb R^n$ can be shown analogously by considering each component separate. Let us remark that the factor $\theta(s,t):=(t-s)/(\log t-\log s)$ is also known as the logarithmic mean.
 \end{ex}
 Besides the above example we have in general $d_{s,t}(x,y)\neq d_{s,t}(y,x)$ as soon as $x\neq y$.
However, it clearly holds $d_{s,t}(x,y)=d_{t,s}(y,x)$, $d_{t,t}(x,y)=d_t(x,y)$ and $d_{s,t}(x,x)=0$.

We will use the following notation: $\partial_t^+u(t):=\limsup_{s\to t}\frac{u(t)-u(s)}{t-s}$.

\begin{Def}
Let  $E: [0,T]\times X\to (-\infty,\infty]$ be a lower semicontinuous functional in $X$.
An absolutely continuous curve $(x_t)_{0\leq t\leq T}$ will be called \emph{dynamic
 EVI$(K,\infty)$-gradient flow} for $E$ 
 if for all $t\in (0,T)$ and all   $y\in Dom(E_t)$ 
  \begin{align}\label{evi-dyn}
     {\frac12\partial_s^+ d^2_{s,t}(x_s,y)}\Big|_{s=t}+\frac K2d_t^2(x_t,y)
       ~\leq E_t(y)-E_t(x_t).
  \end{align}
 We say that the gradient flow 
   $(x_t)_{0\leq t\leq T}$   \emph{starts} in $x'\in X$ if $\lim_{t\searrow0}x_t=x'$.
\end{Def}

We show uniqueness of dynamic EVI$(K,\infty)$ flows by proving a contraction estimate. This estimate involves the logarithmic Lipschitz control $L$ from \eqref{loglip}.
For an estimate without this control see Theorem 7.7 in \cite{sturm2016}.
\begin{lma}
 The following holds true.
 \begin{enumerate}
  \item Suppose that $(x_t)$ is a EVI$(K,\infty)$-gradient flow. Then for every $t\in(0,T)$
  \begin{align}\label{evi0}
   \frac12\partial_s^+d_t^2(x_s,y)_{|s=t}\leq E_t(y)-E_t(x_t)+(L-\frac K2)d_t^2(x_t,y).
  \end{align}
\item There exists at most one EVI$(K,\infty)$-gradient flow starting in $x'$. More precisely the following holds: Let $(x_t)$ and $(y_t)$ be two 
EVI$(K,\infty)$-gradient flows. Then for 
all $s<t$
\begin{align}\label{almostcontraction}
 d_t(x_t,y_t)\leq e^{(3L-K)(t-s)}d_s(x_s,y_s).
\end{align}

 \end{enumerate}

\end{lma}
\begin{proof}
 To show the first assertion note that with $d_{t,s}(y,x_s)=d_{s,t}(x_s,y)$
 \begin{align*}
  \partial_s^+d_{t,s}^2(y,x_s)_{s=t+}&:=\limsup_{s\searrow t}\frac{d_{t,s}^2(y,x_s)-d_t^2(y,x_t)}{s-t}\\
  &\geq \limsup_{s\searrow t}\frac{e^{-2L(s-t)}d_t^2(y,x_s)-d_t^2(y,x_t)}{s-t}\\
  &\geq \limsup_{s\searrow t}\Big\{\frac{d_t^2(y,x_s)-d_t^2(y,x_t)}{s-t}+\frac{(e^{-2L(s-t)}-1)}{s-t}d_t^2(y,x_s)\Big\}\\
  &=\partial_s^+d_t^2(y,x_s)_{|s=t+}+\lim_{s\searrow t}\frac{(e^{-2L(s-t)}-1)}{s-t}d_t^2(y,x_s)\\
  &=\partial_s^+d_t^2(y,x_s)_{|s=t+}-2Ld_t^2(y,x_t),
 \end{align*}
where the first inequality is due to the logarithmic Lipschitz continuity \eqref{loglip}, and the second equality follows from the absolute continuity of 
$(x_t)$. The same argument holds for $\partial_s^+d_{s,t}^2(x_s,y)_{s=t-}:=\limsup_{s\nearrow t}\frac{d_{s,t}^2(x_s,y)-d_t^2(x_t,y)}{s-t}$
replacing $\partial_s^+d_t^2(y,x_s)_{|s=t+}$ by $\partial_s^+d_t^2(y,x_s)_{|s=t-}$, and hence from the EVI$(K,\infty)$ inequality we deduce
 \begin{align*}
   \frac12\partial_s^+d_t^2(x_s,y)_{|s=t}\leq E_t(y)-E_t(x_t)+(L-\frac K2)d_t^2(x_t,y).
  \end{align*}
 In order to show the second assertion, let $(x_t),(y_t)$ be two EVI$(K,\infty)$ gradient flows. Observe that from the absolute continuity of $(x_t)$ and $(y_t)$ 
 it follows that the map $t\mapsto d_t^2(x_t,y_t)$ is absolutely continuous as well.
 This can be seen by applying triangle inequality and \eqref{loglip}. Hence we may write for a.e. $t\in (0,T)$
 \begin{equation}
  \begin{aligned}\label{evi1}
  \frac12\frac{d}{dt}d_t^2(x_t,y_t)\leq &\frac12\limsup_{s\nearrow t}\frac{d_t^2(x_t,y_t)-d_t^2(x_s,y_t)}{t-s}\\
  +&\frac12\limsup_{s\nearrow t}\frac{d_t^2(x_s,y_t)-d_s^2(x_s,y_t)}{t-s}\\
  +&\frac12\limsup_{s\searrow t}\frac{d_t^2(x_t,y_s)-d_t^2(x_t,y_t)}{s-t},
 \end{aligned}
 \end{equation}
 where we used an adaption of \cite[Lemma 4.3.4]{ags}.
Applying \eqref{evi0} and \eqref{loglip} we obtain for a.e. $t\in (0,T)$
\begin{align*}
  \frac12\frac{d}{dt}d_t^2(x_t,y_t)\leq &E_t(y_t)-E_t(x_t)+(L-\frac K2)d_t^2(x_t,y_t)\\
  &+Ld_t^2(x_t,y_t)\\
  &+E_t(x_t)-E_t(y_t)+(L-\frac K2)d_t^2(x_t,y_t)\\
  =& (3L-K)d_t^2(x_t,y_t).
 \end{align*}
We conclude from Gronwall's inequality for a.e. $t>s$
 \begin{align*}
  d_t^2(x_t,y_t)\leq e^{(6L-2K)(t-s)}d_s^2(x_s,y_s).
 \end{align*}
 From the continuity of $t\mapsto d_t(x_t,y_t)$ we obtain that the estimate holds for every $t>s$ and
 in particular we have uniqueness.

\end{proof}

In this general framework it is possible to produce dynamic EVI-gradient flows which are not dynamic EDI-gradient flows as we see in the next
example.
\begin{ex}\label{eviimpliesnoedi}
Let $X=\mathbb R$ and $d_t(x,y)=|x-y|$ for every $t\in[0,T]$, $x,y\in X$. As already seen in Example \ref{exnotdiff}, the energy functional 
$E_t(x)=|x-t|$ is not differentiable at $x_t=t$ for any $t\in[0,T]$. Hence it is not a EDI-gradient flow in the sense of Definition \ref{defedi}.
But it immediately follows from
\begin{align*}
 \frac12\partial_t|x_t-y|^2=(t-y)\leq |y-t|=E_t(y)-E_t(x_t), \qquad \forall y\in X,
\end{align*}
that $(x_t)$ is a EVI$(0,\infty)$-gradient flow.
\end{ex}
We can exclude such behavior if we restrict ourselves to admissible functionals.
\begin{Prop}\label{propevi}
Let $E\colon[0,T]\times X\to\mathbb R$ be an admissible functional, i.e. satisfying the assumptions \textbf{A1, A2, A3} and \textbf{A4} from the previous section.
Let $(x_t)$ be a dynamic EVI$(K,\infty)$-gradient flow for $E$ such that $(x_t)\in AC_{loc}^2([0,T];X)$ and $t\mapsto E_t(x_t)$ is absolutely continuous. Then it is a dynamic EDE-gradient flow as well.
\end{Prop}
\begin{proof}
First note that for a.e. $t$
\begin{align}\label{estmetric}
 \frac12 \partial_s^+d_t^2(x_s,y)_{|s=t}\geq-|\dot x_t|_td_t(x_t,y).
\end{align}
Since $E$ is admissible and $t\to E_t(x_t)$ is supposed to be absolutely continuous it holds for a.e. $t$ 
\begin{equation}\label{inequality1}
\begin{aligned}
 \frac{d}{dt}E_t(x_t)=&(\partial_t E_t)(x_t)+\liminf_{h\to0}\frac{E_t(x_{t+h})-E_t(x_t)}h\\
 =&(\partial_t E_t)(x_t)+\liminf_{h\to0}\frac{E_t(x_{t+h})-E_t(x_t)}{d_t(x_{t+h},x_t)}\frac{d_t(x_{t+h},x_t)}h\\
 \geq&(\partial_t E_t)(x_t)-\limsup_{h\to0}\frac{E_t(x_{t})-E_t(x_{t+h})}{d_t(x_{t+h},x_t)}\frac{d_t(x_{t+h},x_t)}h\\
 \geq &(\partial_t E_t)(x_t)-\frac12|\nabla_tE_t|^2(x_t)-\frac12|\dot x_t|^2_t.
\end{aligned}
\end{equation}
To show the converse inequality recall that $t\mapsto d_t^2(x_t,y)$ is absolutely continuous. Hence, applying the same calculation as in \eqref{evi1}
to the constant curve $y_t\equiv y$, we can write for every $t\in[0,T-h]$ and every $y$ 
\begin{align*}
 \frac12d_{t+h}^2(x_{t+h},y)-\frac12d_t^2(x_t,y)=\frac12\int_t^{t+h}\frac{d}{ds}d_s^2(x_s,y)ds\\
 \leq\int_t^{t+h} E_s(y)-E_s(x_s)+(2L-\frac K2)d_s^2(x_s,y)ds.
\end{align*}

We set $y=x_t$ and find 
\begin{align*}
 \frac12d_{t+h}^2(x_{t+h},x_t)
 \leq &h\int_0^1E_{t+hr}(x_t)-E_{t+hr}(x_{t+hr})dr\\
 +&(2L-\frac K2)\int_t^{t+h}d_r^2(x_r,x_t)dr.
\end{align*}
Again by \eqref{loglip} and the 2-absolute continuity of $(x_t)$ we obtain for some function $g\in L_{loc}^2[0,T]$
\begin{align*}
 \frac12d_t^2(x_{t+h},x_t)\leq e^{2Lh}\left[h\int_0^1E_{t+hr}(x_t)-E_{t+hr}(x_{t+hr})dr+|2L-\frac K2|h^2\int_t^{t+h} g_u^2\, du  \right].
\end{align*}
Dividing by $h^2$ and letting $h\searrow0$, dominated convergence yields
\begin{equation}\label{tdermder}
\begin{aligned}
 \frac12|\dot x_t|_t^2&\leq \int_0^1\lim_{h\searrow0}\frac{E_{t}(x_t)-E_{t+hr}(x_{t+hr})}h+\frac{E_{t+hr}(x_t)-E_{t}(x_{t})}hdr\\
 &=-\frac12\frac{d}{dt}E_t(x_t)+\frac12(\partial_t E_t)(x_t),
\end{aligned}
\end{equation}
for a.e. $t\in(0,T)$. Concerning the slope of $E$ we find that using \eqref{evi0} and \eqref{estmetric}
\begin{equation}\label{slopemder}
\begin{aligned}
 |\nabla_tE_t|(x_t)=&\limsup_{y\to x_t}\frac{[E_t(x_t)-E_t(y)]^+}{d_t(x,y)}\\
 \leq&\limsup_{y\to x_t}\frac{\left[-\partial_s^+d_t^2(x_s,y)_{s=t}+(2L-K)d_t^2(x_t,y)\right]^+}{2d_t(x_t,y)}\\
 \leq&\limsup_{y\to x_t}\frac{\left[2|\dot x_t|_td_t(x_t,y)+(2L-K)d_t^2(x_t,y)\right]^+}{2d_t(x_t,y)}
 \leq |\dot x_t|_t,
\end{aligned}
\end{equation}
for almost every $t$.
Combining \eqref{tdermder} and \eqref{slopemder} we conclude
\begin{equation}\label{inequality2}
\begin{aligned}
 \frac{d}{dt}E_t(x_t)\leq (\partial_tE_t)(x_t)-|\dot x_t|_t^2\\
 \leq(\partial_tE_t)(x_t)-\frac{|\dot x_t|_t^2}2-\frac{|\nabla_tE_t|^2(x_t)}2.
\end{aligned}
\end{equation}
We obtain \eqref{ede} from \eqref{inequality1} and \eqref{inequality2} after integrating on the interval $(0,t)$.
\end{proof}

\subsection{Existence of dynamic EDI-gradient flows}\label{sec:gradflow}
We are interested in the following problem.
\begin{problem}
 Given a function $E\colon[0,T]\times X\to (-\infty,+\infty]$, and an initial value $\bar x\in \Dom(E)$, find an EDI-gradient flow $(x_t)$ for $E$.
\end{problem}
Under suitable topological assumptions we will find a gradient flow for a certain class of energy functionals using the minimizing movement scheme, which we describe 
in the subsequent sections, cf. \cite{ags}.
\subsubsection*{Topological assumptions}
We additionally impose a topology $\sigma$ on $X$ such that $\sigma$ is weaker than the topology induced by $(d_t)$ and $d_t$ is sequentially $\sigma$-lower
semicontinuous, i.e.
\begin{equation*}
 \text{if } x_n\overset{\sigma}\rightharpoonup x\text{ and }y_n\overset\sigma\rightharpoonup y, \text{ then } \liminf_{n\to\infty} d_t(x_n,y_n)\geq d_t(x,y) \text{ for every }t\in[0,T].
\end{equation*}
Let $E\colon[0,T]\times X\to(-\infty,\infty]$ be a functional satisfying \textbf{A1}, \textbf{A2}, and \textbf{A4}.
We will extend our assumptions by the following.
\begin{description}
 \item[A5] If $(x_n)\subset X$ with $\sup_{n,m}d_t(x_n,x_m)<\infty$, then $(x_n)$ admits a $\sigma$-convergent subsequence.
  \item[A3$^*$]  For each $t\in[0,T]$, $x\mapsto E_t(x)$ is sequentially $\sigma$-lower semicontinuous.

 \end{description}

\subsubsection*{Approximation}
We fix a time step $h>0$ and subdivide the interval $[0,T]$ into the partition
\begin{align*}
 \mathcal{P}_h:=\{t_0=0<t_1<\cdots<t_{N-1}<T\leq t_N\},\qquad t_n=nh, N\in\mathbb{N}.
\end{align*}
For $0\leq t\leq T$ we define the piecewise constant interpolants $\overline{h}(t)$ and $\underline{h}(t)$ associated with the partition $\mathcal{P}_h$ 
in the following way;
\begin{equation}\label{dicretetime}
\overline{h}(0)=0=\underline{h}(0), \text{ and for }t\in (t_{n-1},t_n] \quad \overline{h}(t)=t_n , \quad \underline{h}(t)=t_{n-1}.
\end{equation}
The definition implies that $\overline{h}(t)\searrow t$ and $\underline{h}(t)\nearrow t$ if $h\searrow0$.

For a given initial value $\bar x$ we recursively define a sequence $(x_n^h)$ of minimizers by
\begin{align}\label{minprob}
 x_0^h:=\bar x, \qquad x_n^h:=\arg\min_{x}\Big\{E_{t_n}(x)+\frac1{2h}d_{t_n}^2(x,x_{n-1}^h)\Big\}
\end{align}
\begin{Prop}\label{minexist}
For every $\bar x\in \Dom(E)$ and $h>0$ there exists a solution to the minimization problem \eqref{minprob}.
 
\end{Prop}

\begin{proof}
Existence follows by the direct method of calculus. Define 
\begin{equation*}
 \phi(h,\bar x,t;\cdot):=E_{t}(\cdot)+\frac1{2h}d^2_t(\bar x,\cdot).
\end{equation*}
Since $E$ is uniformly bounded from below we may take a minimizing sequence $(x_n)_{n\in\mathbb{N}}$ such that $d_t^2(x_n,\bar x)$ remains bounded uniformly in $n$. From the triangle inequality we deduce that $\sup_{n,m} d_t(x_n,x_m)<\infty$. 
Hence \textbf{A5}
guarantees existence of a $\sigma$-convergent subsequence $x_{n_k}$. The weak limit point $x\in \Dom(E)$ is a minimizer of 
$\phi(h,\bar x,t;\cdot)$, which is due to the $\sigma$-lower semicontinuity of the distance and the functional.
\end{proof}

\begin{Def}
Fix $h>0$ and
 let $s\in[0,T-h]$. For $0<r<T-s$ define
 \begin{align}
  &J_{s,r}(y):=\min_x\Big\{E_{s+r}(x)+\frac1{2r}d_{s+h}^2(x,y)\Big\},\\
  &A_{s,r}(y):=\arg\min_x\Big\{E_{s+r}(x)+\frac1{2r}d_{s+h}^2(x,y)\Big\}.
 \end{align}

\end{Def}

\begin{lma}\label{slopeest}
For $x_r\in A_{s,r}(y)$ we have 
\begin{equation*}
|\nabla_{s+h}E_{s+r}|(x_r)\leq\frac1rd_{s+h}(x_r,y)
\end{equation*}
and for $0<r_1<r_2<T-s$ 
\begin{equation}\label{mondist}
d^2_{s+h}(x_{r_1},y)\leq d^2_{s+h}(x_{r_2},y)+4r_1r_2L^*,
\end{equation}
where $L^*$ is the Lipschitz constant in assumption \textbf{A4}.
\end{lma}
\begin{proof}
By optimality of $x_r$ we have for every $x\in X$
\begin{align*}
&\frac{E_{s+r}(x_r)-E_{s+r}(x)}{d_{s+h}(x_r,x)}\leq \frac{d^2_{s+h}(x,y)-d^2_{s+h}(x_r,y)}{2rd_{s+h}(x_r,x)}\\
=&\frac{(d_{s+h}(x,y)-d_{s+h}(x_r,y))(d_{s+h}(x,y)+d_{s+h}(x_r,y))}{2rd_{s+h}(x_r,x)}\\
\leq &\frac{(d_{s+h}(x,y)+d_{s+h}(x_r,y))}{2r}
\end{align*}
Taking the limsup as $x\to x_r$ we get the first assertion.
To show the second assertion note that on the one hand we have
\begin{equation*}
E_{s+r_1}(x_{r_1})+\frac1{2r_1}d^2_{s+h}(x_{r_1},y)\leq E_{s+r_1}(x_{r_2})+\frac1{2r_1}d^2_{s+h}(x_{r_2},y),
\end{equation*}
and on the other
\begin{equation*}
E_{s+r_2}(x_{r_2})+\frac1{2r_2}d^2_{s+h}(x_{r_2},y)\leq E_{s+r_2}(x_{r_1})+\frac1{2r_2}d^2_{s+h}(x_{r_1},y).
\end{equation*}
Adding these two inequalities, using the Lipschitz property of $t\mapsto E_t(x)$ and dividing by $\frac1{2r_1}-\frac1{2r_2}$ yields \eqref{mondist}.
\end{proof}
\begin{lma}
 The map $r\mapsto J_{s,r}(y)$ is locally Lipschitz and for almost every $r\in(0,T-s)$ we have for $x_r\in A_{s,r}(y)$
 \begin{align}\label{diffofmin}
  \frac{d}{dr}J_{s,r}(y)=-\frac1{2r^2}d_{s+h}^2(x_r,y)+ (\partial_rE_{s+r})(x_r).
 \end{align}

\end{lma}
\begin{proof}
 Fix $0<r_1<r_2<T-s$. Then
 \begin{equation}
 \begin{aligned}\label{jlipschitz}
  J_{s,r_2}(y)-J_{s,r_1}(y)=&E_{s+r_2}(x_{r_2})-E_{s+r_1}(x_{r_1})\\
  &+\frac1{2r_2}d^2_{s+h}(x_{r_2},y)-\frac1{2r_1}d^2_{s+h}(x_{r_1},y)\\
  \leq& E_{s+r_2}(x_{r_1})-E_{s+r_1}(x_{r_1})+\frac{r_1-r_2}{2r_2r_1}d^2_{s+h}(x_{r_1},y)\\
  \leq& L^*(r_2-r_1)-\frac{r_2-r_1}{2r_2r_1}d^2_{s+h}(x_{r_1},y),
 \end{aligned}
 \end{equation}
 where $L^*$ denotes the Lipschitz constant from \textbf{A4}.
Conversely, changing the roles of $x_{r_1}$ and $x_{r_2}$, we obtain
\begin{align*}
 J_{s,r_2}(y)-J_{s,r_1}(y)\geq -L^*(r_2-r_1)-\frac{r_2-r_1}{2r_2r_1}d^2_{s+h}(x_{r_2},y).
\end{align*}
Combining these two inequalities yields
\begin{align*}
 |J_{s,r_2}(y)-J_{s,r_1}(y)|\leq L^*|r_2-r_1|+\frac{|r_2-r_1|}{2r_1r_2}d^2_{s+h}(x_{r_2},y),
\end{align*}
which means $r\mapsto J_{s,r}(y)$ is locally Lipschitz.
Dividing by $r_2-r_1$ and letting $r_1\to r_2$ in \eqref{jlipschitz} yields on the one hand for the left derivative
\begin{align*}
 \frac{d^-}{dr}J_{s,r}(y)\leq -\frac1{2r^2}d^2_{s+h}(x_r,y)+(\partial_rE_{s+r})(x_r),
\end{align*}
for every differentiability point $r$ of $r\mapsto E_{t+r}$.
On the other hand we obtain similarly for the right derivative
\begin{align*}
 \frac{d^+}{dr}J_{s,r}(y)\geq -\frac1{2r^2}d^2_{s+h}(x_r,y)+(\partial_rE_{s+r})(x_r),
\end{align*}
for every differentiability point of $r\mapsto E_{t+r}$.
By local Lipschitz continuity we have for a.e. $0<r<T-s$
\begin{align*}
 \frac{d}{dr}J_{s,r}(\nu)= -\frac1{2r^2}d^2_{s+h}(x_r,y)+(\partial_rE_{s+r})(x_r).
\end{align*}
\end{proof}

\begin{lma}
For $s\in [0,T]$ and $0<r_1<r_2<T-s$
 \begin{align}
  E_{s}(y)\geq J_{s,r_1}(y)-Cr_1\geq J_{s,r_2}(y)-Cr_2\label{monoton}
 \end{align}
 \begin{align}
  \lim_{r\to0} d_{s+h}(y,x_r)=0\label{continuity} \text{ if }y\in \Dom(E).
 \end{align}
 In particular $\lim_{r\to0} J_{s,r}(y)=E_s(y)$.

\end{lma}
\begin{proof}
The first inequality in \eqref{monoton} directly follows from
\begin{equation*}
 E_{s+r_1}(x_{r_1})+\frac1{2r_1}d^2_{s+h}(x_{r_1},y)\leq E_{s+r_1}(y)\leq E_{s}(y)+L^*r_1.
\end{equation*}
The second one follows by
\begin{align*}
 E_{s+r_1}(x)+\frac1{2r_1}d^2_{s+h}(x,y)&\geq E_{s+r_1}(x)+\frac1{2r_2}d^2_{s+h}(x,y)\\
 &\geq  E_{s+r_2}(x)+\frac1{2r_2}d^2_{s+h}(x,y)-L^*(r_2-r_1),
\end{align*}
and minimizing over all $x$.
Since for every $x\in \Dom(E)$
\begin{align*}
 0\leq d^2_{s+h}(y,x_r)&\leq-2rE_{s+r}(x_r)+d^2_{s+h}(y,x)+2rE_{s+r}(x)\\
 &\leq-2r\inf E+d^2_{s+h}(y,x)+2rE_{s+r}(x).
\end{align*}
Passing to the limit $r\to0$
\begin{align*}
 \lim_{r\to0}d^2_{s+h}(x_r,y)\leq d^2_{s+h}(x,y) \text{ for every }x\in \Dom(E).
\end{align*}
Since $y\in \Dom(E)$ we conclude \eqref{continuity}.
To check the last one we combine \eqref{monoton} with the lower semicontinuity of $x\mapsto E_t(x)$,
\begin{equation*}
 E_t(y)\geq \limsup_{r\to 0} J_{t,r}(y)\geq \liminf_{r\to 0} E_{t+r}(x_r)\geq  E_t(y).
\end{equation*}

\end{proof}

\begin{cor}
For every $0<r_0<T-s$ we have
\begin{equation}
\begin{aligned}\label{intofdiff}
&E_{s+r_0}(x_{r_0})+\frac1{2r_0}d^2_{s+h}(x_{r_0},y)\\
=
E_s(y)&-\int_0^{r_0}\frac1{2r^2}d_{s+h}^2(x_r,y)dr+\int_0^{r_0}(\partial_rE_{s+r})(x_r)dr.
\end{aligned}
\end{equation}
\end{cor}
\begin{proof}
Integrate \eqref{diffofmin} from 0 to $r_0$ and use that $\lim_{r\to0} J_{s,r}(y)=E_s(y)$.
\end{proof}

In the following we introduce dynamic counterparts for the variational interpolation, the discrete speed and the discrete slope, cf. \cite{ags,rms}.
\begin{Def}\label{interpolants}
Let $\bar x\in \Dom(E)$ be the initial value and $x_n^h$ be a sequence defined by the minimization problem \eqref{minprob}. 
A \emph{discrete solution} is a curve $t\mapsto\bar x_t^h$ defined by
\begin{align*}
\bar x_t^h=x_n^h, \text{ for }t\in(t_{n-1},t_n],
\end{align*}
and $\bar x_0^h=\bar x$.\\
A \emph{variational interpolation} is a map $t\to\tilde x_t^h$ defined by
\begin{align*}
\tilde x_t^h=\arg\min\Big\{E_{t}(x)+\frac1{2r}d^2_{t_n}(x,x_{n-1}^h)\Big\},\\
\text{for }t=t_{n-1}+r\in(t_{n-1},t_n],
\end{align*}
and $\tilde x_0^h=\bar x$.\\
We define the \emph{discrete speed} $\Dsp^h\colon[0,T]\to[0,\infty)$ and the \emph{discrete slope} $\Dsl^h\colon[0,T]\to[0,\infty)$ in the following way
\begin{align*}
\Dsp_r^h&=\frac1hd_{t_n}(\bar x_{t_n}^h,\bar x_{t_{n-1}}^h), \quad r\in(t_{n-1},t_n],\\
\Dsl_r^h&=\frac1{(r-t_{n-1})}d_{t_n}(\bar x_{t_{n-1}}^h,\tilde x_r^h), \quad r\in(t_{n-1},t_n].
\end{align*}
\end{Def}
Note that $\tilde x_{t_n}^h=x_n^h=\bar x_{t_n}^h$.

\begin{Prop}\label{discreteede}
We have for $0\leq s\leq t\leq T$
\begin{equation}
\begin{aligned}\label{eq:discreteede}
E_{\overline{h}(t)}(\bar x_{t}^h)&+\frac12\int_{\overline{h}(s)}^{\overline{h}(t)}(\Dsp_r^h)^2dr+\frac12\int_{\overline{h}(s)}^{\overline{h}(t)}(\Dsl_r^h)^2dr\\
=&E_{\overline{h}(s)}(\bar x_{s}^h)+\int_{\overline{h}(s)}^{\overline{h}(t)}(\partial_rE_r)(\tilde x_r^h)dr.
\end{aligned}
\end{equation}

\end{Prop}
\begin{proof}
Let $t\in(t_{n-1},t_n]$. We want to apply equation \eqref{intofdiff} with $s=t_{n-1}$, $r_0=t-s$, $y=x_{n-1}^h$. Then with
$x_{r_0}=\tilde x_t^h$ and $x_r=\tilde x_{t_{n-1}+r}^h$  we find
\begin{align*}
E_t(\tilde x_t^h)&+\frac1{2(t-t_{n-1})}d^2_{t_n}(\tilde x_t^h, x_{n-1}^h)+\int_{t_{n-1}}^{t}\frac1{2(r-t_{n-1})^2}d^2_{t_n}(x_{n-1}^h,\tilde x_r^h)dr\\
=&E_{t_{n-1}}(x_{n-1}^h)+\int_{t_{n-1}}^t(\partial_r E_r)(\tilde x_r^h)dr.
\end{align*}
For $t=t_n$ we obtain
\begin{equation} 
\begin{aligned}\label{eq:energy}
E_{t_n}(\bar x_{t_n}^h)&+\frac1{2h^2}\int_{t_{n-1}}^{t_n}d^2_{t_n}(\bar x_{t_n}^h,\bar x_{t_{n-1}}^h)dr+\int_{t_{n-1}}^{t_n}\frac1{2(r-t_{n-1})^2}
d^2_{t_n}(\bar x_{t_{n-1}}^h,\tilde x_r^h)dr\\
=&E_{t_{n-1}}(\bar x_{t_{n-1}}^h)+\int_{t_{n-1}}^{t_n}(\partial_rE_r)(\tilde x_r^h)dr.
\end{aligned}
\end{equation}
Summing up from $n+1$ to $m$ yields
\begin{align*}
&E_{t_m}(\bar x_{t_m}^h)+\frac1{2h^2}\sum_{j=n+1}^m\int_{t_{j-1}}^{t_j}d^2_{t_j}(\bar x_{t_j}^h,\bar x_{t_{j-1}}^h)dr\\
+&\sum_{j=n+1}^m\int_{t_{j-1}}^{t_j}\frac1{2(r-t_{j-1})^2}d^2_{t_j}(\bar x_{t_{j-1}}^h,\tilde x_r^h)dr
=E_{t_{n}}(\bar x_{t_{n}}^h)+\int_{t_{n}}^{t_m}(\partial_rE_r)(\tilde x_r^h)dr.
\end{align*}
Now plugging in the definitions of the discrete slope and the discrete speed respectively
\begin{align*}
E_{t_m}(\bar x_{t_m}^h)&+\frac12\int_{t_{n}}^{t_m}(\Dsp_r^h)^2dr+\frac12\int_{t_{n}}^{t_m}(\Dsl_r^h)^2dr\\
=&E_{t_{n}}(\bar x_{t_{n}}^h)+\int_{t_{n}}^{t_m}(\partial E_r)(\tilde x_r^h)dr,
\end{align*}
which shows \eqref{eq:discreteede}.

\end{proof}

\begin{remark}\label{ede:degiorgi}
Alternatively, for $t\in(t_{n-1},t_n]$ we can write
\begin{align*}
 E_t(\tilde x_t^h)&+\frac1{2(t-t_{n-1})}d^2_{t_n}(\tilde x_t^h,\tilde x_{t_{n-1}}^h)+\frac12\int_0^{t_{n-1}}(\Dsp_r^h)^2dr+\frac12\int_0^t(\Dsl_r^h)^2dr\\
 &=E_0(\bar x)+\int_0^t(\partial_rE_r)(\tilde x_r^h)dr.
\end{align*}

\end{remark}
The following proposition provides essential a priori bounds.
\begin{Prop}
There exist constants $C_1$, $C_2$, $C_3$ such that for all \\
$0\leq t,mh\leq T$
\begin{equation}\label{eq1:apest}
E_{t}(\tilde x_{t}^h)\leq C_1,
\end{equation}
\begin{equation}\label{eq2:apest}
\frac1{2h}\sum_{n=1}^md_{t_n}^2(\bar x_{t_n}^h,\bar x_{t_{n-1}}^h)\leq C_2,
\end{equation}
\begin{equation}\label{eq3:apest}
d^2_{t^*}(\tilde x_t^h,\bar x_t^h)\leq C_3h, \text{ for some fixed }t^*.
\end{equation}
\end{Prop}

\begin{proof}
From Remark \ref{ede:degiorgi} we deduce 
\begin{equation*}
 E_{t}(\tilde x_{t}^h)\leq E_0(\bar x)+L^*T,
\end{equation*}
which shows \eqref{eq1:apest}.

We drop the nonnegative slope term in equation \eqref{eq:energy} to obtain
\begin{align*}
\frac1{2h}d^2_{t_n}(\bar x_{t_n}^h,\bar x_{t_{n-1}}^h)
\leq E_{t_{n-1}}(\bar x_{t_{n-1}}^h)-E_{t_n}(\bar x_{t_n}^h)+\int_{t_{n-1}}^{t_n}(\partial_rE_r)(\tilde x_r^h)dr.
\end{align*}
Summing up to $m$ and applying the Lipschitz property of $t\mapsto E_t$
\begin{align*}
\frac1{2h}\sum_{n=1}^md_{t_n}^2(\bar x_{t_n}^h,\bar x_{t_{n-1}}^h)&\leq E_{t_0}(\bar x_{t_0}^h)
-E_{t_m}(\bar x_{t_m}^h)+\int_{t_0}^{t_{m-1}}(\partial_rE_r)(\tilde x_r^h)dr\\
&\leq  E_{t_0}(\bar x_{t_0}^h)
-E_{t_m}(\bar x_{t_m}^h)+T L^*,
\end{align*}
we obtain on the one hand
\begin{equation*}
E_{t_m}(\bar x_{t_m}^h)\leq  E_{t_0}(\bar x_{t_0}^h)+T L^*,
\end{equation*}
and since $\inf E(x)>-\infty$
\begin{equation*}
\frac1{2h}\sum_{n=1}^md_{t_n}^2(\bar x_{t_n}^h,\bar x_{t_{n-1}}^h)\leq C_2.
\end{equation*}
To show \eqref{eq3:apest} note that for $t\in (t_{n-1},t_n]$
\begin{align*}
d^2_{t_n}(\tilde x_t^h,\bar x_t^h)&=d^2_{t_n}(\tilde x_t^h,x_n^h)\leq 2d^2_{t_n}(\tilde x_t^h,x_{n-1}^h)+2d^2_{t_n}(x_{n-1}^h,x_n^h)\\
&\leq 4d^2_{t_n}(x_n^h,x_{n-1}^h)+8(t-t_{n-1})hC,
\end{align*}
where the third inequality is a consequence of \eqref{mondist}. Applying \eqref{eq2:apest} and \eqref{loglip} we conclude \eqref{eq3:apest}.
\end{proof}

\begin{Prop}\label{compactness}
There exist a subsequence $h_n$ with $\lim_n h_n=0$, a curve $(x_t)\subset \AC^2([0,T];X)$ and a function $A\in L^2(0,T)$ such that 
\begin{equation*}
\bar x_t^{h_n}\overset{\sigma}\rightharpoonup x_t, \quad \tilde x_t^{h_n}\overset{\sigma}\rightharpoonup x_t\text{ for all }t,
\end{equation*}
\begin{equation*}
\text{ and } |\Dsp^{h_n}|\rightharpoonup A \text{ weakly in }L^2(0,T).
\end{equation*}
Further $|\dot x|_t\leq A(t)$ holds almost everywhere.

\end{Prop}
\begin{proof}
We want to apply a refined version of Arzel\`a-Ascoli \cite[Proposition 3.3.1]{ags} to the family $(\bar x^h)_{h>0}$.
Owing to the estimates \eqref{loglip} and \eqref{eq2:apest} we have
\begin{equation*}
 d_{t^*}(\bar x_t^h,\bar x)\leq tC_2e^C,
\end{equation*}
and together with  \textbf{A5} this yields that the curves $\bar x^h\colon[0,T]\to X$ take values in a $\sigma$-sequentially compact set.
From the estimate \eqref{eq2:apest} we further deduce
\begin{align*}
\int_t^s|\Dsp_r^h|^2dr \leq \sum_{j=t_{n}}^{t_m}\frac1h d^2_{t_j}(\bar x_{t_j}^h,\bar x_{t_{j-1}}^h)\leq 2C_2,
\end{align*}
for $\overline{h}(s)=t_m$, $\underline{h}(t)=t_n$. Applying the Banach Alaoglu Theorem we can extract a subsequence $h_n$ and a function 
$A\in L^2(0,T)$ such that $|\Dsp^{h_n}|\rightharpoonup A$ weakly in $L^2([0,T])$.
For fixed $t^*$ and $s<t$ we deduce from the log-Lipschitz property \eqref{loglip}
\begin{align*}
d_{t^*}(\bar x_t^h,\bar x_s^h)\leq\int_{\underline{h}(s)}^{\overline{h}(t)} \frac1hd_{t^*}(\bar x_r^h,\bar x_{r-h}^h)dr\\
\leq
\int_{\underline{h}(s)}^{\overline{h}(t)}\frac1hd_{\overline{h}(r)}(\bar x_r^h,\bar x_{r-h}^h)e^{L|\overline{h}(r)-t^*|}dr,
\end{align*}
and hence
\begin{align*}
 \limsup_{n\to\infty}d_{t^*}(\bar x_t^{h_n},\bar x_s^{h_n})\leq\int_t^s A(r)e^{L|r-t^*|}dr.
\end{align*}
Propostion 3.3.1 in \cite{ags} and \eqref{eq3:apest} imply that there exists a further subsequence, not relabeled, and a limit curve $x:[0,T]\to X$ such that
\begin{equation*}
 \bar x_t^{h_n}\overset{\sigma}\rightharpoonup  x_t, \quad  \tilde x_t^{h_n}\overset{\sigma}\rightharpoonup  x_t\quad \forall t\in[0,T].
\end{equation*}
This curve is absolutely continuous since 
\begin{equation*}
d_{t^*}(x_t,x_s)\leq\liminf_{n\to\infty} d_{t^*}(\bar x_t^{h_n},\bar x_s^{h_n})
\leq\int_s^tA(r)e^{L|r-t^*|}dr,
\end{equation*}
In particular if we take $t^*=t$ in the argumentation above the Lebesgue differentiation theorem implies that
\begin{equation*}
|\dot x|_t\leq \limsup_{s\nearrow t}\frac1{t-s}\int_s^tA(r)e^{L|r-t|}dr\leq A(t)
\end{equation*}
holds true for almost every $t$.

\end{proof}

\begin{Prop}\label{stability}
 Suppose additionally to our standing assumptions \textbf{A1}, \textbf{A2}, \textbf{A3$^*$}, \textbf{A4} and \textbf{A5} that 
 \begin{itemize}
 \item If $x_n \overset{\sigma}\rightharpoonup x$ as $n\to \infty$ then 
  \begin{equation}\label{uscofder}
   \limsup_{n\to\infty} \partial_tE_t(x_n)\leq\partial_tE_t(x),
  \end{equation}

  \item if $t_n\to t$ and $x_n\overset{\sigma}\rightharpoonup x$, then
  $$|\nabla_t E_t|^2(x)\leq\liminf |\nabla_{t_n}E_{t}|^2(x_n).$$
 \end{itemize}
 Then every limit curve $(x_t)$ from Proposition \ref{compactness} satisfies the EDI formula
 \begin{equation}
 \begin{aligned}
  E_t(x_t)+\frac12\int_0^t|\dot x|^2_rdr+\frac12\int_0^t|\nabla_rE_r|^2(x_r)dr
  \leq E_0(\bar x)+\int_0^t(\partial_r E_r)(x_r)dr,
 \end{aligned}
 \end{equation}
 for every $t\in[0,T]$.

\end{Prop}

\begin{proof}
 Recall that Proposition \ref{discreteede} states for $s=0$
 \begin{align*}
E_{\overline{h_n}(t)}(\bar x_{t}^{h_n})+\frac12\int_{0}^{\overline{h_n}(t)}(\Dsp_r^{h_n})^2dr+\frac12\int_{0}^{\overline{h_n}(t)}(\Dsl_r^{h_n})^2dr\\
=E_{0}(\bar x)+\int_{0}^{\overline{h_n}(t)}(\partial_rE_r)(\tilde x_r^{h_n})dr.
\end{align*}
Since both $\bar x_{t}^{h_n}$, $\tilde x_{t}^{h_n}$ $\sigma$-converges to $x_t$ for every $t$, $x\mapsto E_t(x)$ 
is $\sigma$-lower semicontinuous, $x\mapsto \partial_tE_t(x)$ is $\sigma$-upper semicontinuous and $t\to E_t(x)$ is Lipschitz continuous uniformly in $x$, we know
\begin{equation*}
 \liminf_{n\to\infty}E_{\overline{h_n}(t)}(\bar x_{t}^{h_n})\geq E_t(x_t),
\end{equation*}
and
\begin{equation*}
 \int_0^t(\partial_rE_r)(x_r)dr\geq\int_0^t\limsup(\partial_rE_r)(\tilde x_r^{h_n})dr
 \geq\liminf\int_0^t (\partial_rE_r)(\tilde x_r^{h_n})dr,
\end{equation*}
where the last inequality follows from Fatou's Lemma.
From Proposition \ref{compactness} and Lemma \ref{slopeest} we deduce
\begin{equation*}
 \int_0^t|\dot x|_r^2dr\leq\int_0^t A(r)^2dr\leq \liminf_{n\to \infty}\int_0^t (Dsp_r^{h_n})^2dr,
\end{equation*}
and
\begin{equation*}
 \int_0^t|\nabla_rE_r|^2(x_r)dr\leq\liminf\int_0^t|\nabla_{\overline{h_n}(r)}E_{r}|^2(\tilde x_r^{h_n})dr\leq \liminf\int_0^t(\Dsl_r^{h_n})^2dr
\end{equation*}
Combining these inequalities with \eqref{eq:discreteede}  we conclude
\begin{align*}
 &E_t(x_t)+\frac12\int_0^t|\dot x|^2_rdr+\frac12\int_0^t|\nabla_rE_r|^2(x_r)dr\\
 \leq &\liminf \left[E_{\overline{h_n}(t)}(\bar x_{t}^{h_n})+\frac12\int_0^{\overline{h_n}(t)} (Dsp_r^{h_n})^2dr 
 +\frac12\int_0^{\overline{h_n}(t)}(\Dsl_r^{h_n})^2dr\right]\\
 \leq & \liminf\left[E_0(\bar x)+\int_0^{\overline{h_n}(t)}(\partial_rE_r)(\tilde x_r^{h_n})dr\right]\\
 \leq& E_0(\bar x)+\int_0^t(\partial_r E_r)(x_r)dr,
\end{align*}
which is the assertion.

\end{proof}

\section{Dynamic gradient flow of the entropy}\label{secent}
In this section we want to study gradient flows for the Boltzmann entropy on probability space, where the metric of the space and the reference measure of the entropy varies in 
time.
To show existence we apply the results from Section \ref{sec:gradflow}. We then go on to show also uniqueness.

Let $X$ be a topological space equipped with a family of complete separable geodesic metrics $(d_t)_{t\in[0,T]}$ satisfying \eqref{loglip} and
a Borel probability measure $m$.
We define $\mathcal{P}(X)$ to be the space of Borel probability measures on $X$ and we denote the subspace of probability measures absolutely continuous 
to the measure $m$ by $\mathcal{P}^{ac}(X)$. Further let $\mathcal P_2(X)$ be the space of probability measures with finite
second moments on $X$
\begin{align*}
 \mathcal P_2(X):=\Big\{\mu\in\mathcal P(X)\Big|&\int d_t^2(x,x_0)d\mu(x)<\infty\\
 &\text{ for some, and thus any, }x_0\in X,t\in[0,T]\Big\}.
\end{align*}

We say that a sequence $\mu_n\subset\mathcal{P}(X)$ converges weakly to $\mu$ if
 $\lim\int_Xf d\mu_n=\int_Xfd\mu$ for every $f\in\mathcal{C}_b^0(X)$.
We say that a sequence $\rho_n\subset L^1(X,m)$ converges weakly to $\rho$ if
 $\lim\int_Xf\rho_n dm=\int_Xf\rho dm$ for every $f\in L^\infty(X,m)$. Note that if $\rho_n$ converges weakly to $\rho$ in $L^1(X,m)$
 then $\mu_n=\rho_nm$ converges weakly to $\mu=\rho m$ in $\mathcal{P}(X)$.\\
 
\subsection{Time-dependent Kantorovich metrics}
For every metric $d_t$ we define the $L^2$-Kantorovich distance $W_t$ on the space $\mathcal{P}_2(X)$:
\begin{equation*}
 W_t(\mu,\nu)=\inf\{C_t(\gamma): \pi^1_\#\gamma=\mu, \pi^2_\#\gamma=\nu\}^{1/2},
\end{equation*}
where $C_t(\gamma)$ is the cost of the plan $\gamma\in\mathcal{P}(X\times X)$
\begin{equation*}
 C_t(\gamma)=\int d_t^2(x,y)d\gamma(x,y),
\end{equation*}
and $\pi_\#^i\gamma$ denote the first and second marginals of $\gamma$.

For each $t\in[0,T]$, $(\mathcal P_2(X),W_t)$ is a geodesic Polish space, see e.g. \cite{villani2009,ag}.
It is well-known that convergence in the $L^2$-Kantorovich distance $W_t$ implies weak convergence in $\mathcal{P}(X)$
and that $W_t$ is lower semicontinuous on $\mathcal{P}(X)$ (cf. \cite[Theorem 6.8]{villani2009} and \cite[Remark 6.10]{villani2009}).
The bound \eqref{loglip} is equivalent to 
\begin{equation}\label{wloglip}
|\log W_t(\mu,\nu)/W_s(\mu,\nu)|\leq L|t-s|,
\end{equation}
for all $s,t$ and all probability measures on $X$, see Lemma 2.1 in \cite{sturm2015}. 

The convexity of the squared metric speed is crucial for showing uniqueness of the gradient flow. More precisely we have the following result 
\cite[Lemma 14]{gigli2009}.
\begin{lma}\label{cofsquaredslope}
Let $(\mu_t^1) ,(\mu_t^2)\in \AC^2([0,T];\mathcal P_2(X))$ be two absolutely continuous curves. Define $\mu_t^{1,2}=(\mu_t^1+\mu_t^2)/2$. Then $(\mu_t^{1,2})$ is absolutely continuous and the 
following bound on its metric derivative holds
\begin{equation*}
 |\dot\mu^{1,2}|_t^2\leq\frac{|\dot\mu^1|_t^2+|\dot\mu^2|^2_t}{2}.
\end{equation*}
\end{lma}
\begin{proof}
 Fix $s,t\in[0,T]$. Pick optimal plans $\gamma^1,\gamma^2$, which minimize $W_t(\mu_t^1,\mu_s^1)$ and $W_t(\mu_t^2,\mu_s^2)$ respectively. Then the plan 
 $(\gamma^1+\gamma^2)/2$ has marginals $\mu_t^{1,2}$ and $\mu_s^{1,2}$ and therefore it holds
 \begin{align*}
  W_t^2(\mu_t^{1,2},\mu_s^{1,2})\leq\int d_t^2(x,y)\,d\frac{(\gamma^1+\gamma^2)}2(x,y)\\
  =\frac12\int d_t^2(x,y)\,d\gamma^1(x,y)+\frac12\int d_t^2(x,y)\,d\gamma^2(x,y)\\
  =\frac12W_t^2(\mu_t^1,\mu_s^1)+\frac12W_t^2(\mu_t^2,\mu_s^2).
 \end{align*}
Thus the curve $(\mu_t^{1,2})$ is absolutely continuous. Dividing by $(s-t)^2$ and taking the superior limit as $s$ goes to $t$ we get for its speed
\begin{equation*}
|\dot\mu^{1,2}|_t^2\leq\frac{|\dot\mu^1|_t^2+|\dot\mu^2|^2_t}{2}. 
\end{equation*}

\end{proof}

\subsection{Time-dependent Boltzmann entropy}
We consider a family of measures $(m_t)_{t\in[0,T]}$ on $X$. We suppose that for every $t\in[0,T]$ there exists a 
function $f_t\in L^\infty(X,m)$ such that $m_t=e^{-f_t}m$. Moreover let us always assume that there exists
a constant $L^*$ such that 
\begin{equation}\label{eqintimeandspace}
 |f_t(x)-f_s(x)|\leq L^*|t-s|
\end{equation}
for all $s,t$ and all $x$.

We denote by $S_t$ the relative Boltzmann entropy with respect to $m_t$,
\begin{align*}
& S\colon[0,T]\times\mathcal{P}_2(X)\to[-\infty,\infty],\\
 (t,\mu)\mapsto &S_t(\mu)=\Ent(\mu|m_t)=\int\rho\log\rho\,dm_t,
\end{align*}
where $\rho=\,d\mu/\,d m_t$ provided that $\mu\ll m_t$. Otherwise we set $S_t(\mu)=\infty$.
It follows directly from the representation of the measures $m_t$ that
$$S_t(\mu)=\Ent(\mu)+\int f_t(x)d\mu(x),$$
where $\Ent(\mu)=\Ent(\mu|m)$.

In the next lemma we list the crucial properties of the relative entropy functional.
\begin{lma}\label{ent}
The entropy $S\colon[0,T]\times\mathcal{P}_2(X)\to[-\infty,\infty]$ satisfies \textbf{A1}, \textbf{A2}, \textbf{A3$^*$} and \textbf{A4}, i.e.
\begin{enumerate}
  \item The domain $\Dom(S_t)$ is time-independent.
  \item $S_t(\mu)$ is uniformly bounded from below.
  \item For each $t\in[0,T]$, $\mu\mapsto S_t(\mu)$ is lower semicontinuous with respect to weak convergence over probability space.
  \item For every $\mu\in \Dom(S)$ the map $t\mapsto S_t(\mu)$ is Lipschitz continuous with Lipschitz constant $L^*$ and for the derivative it holds
  \begin{equation*}
   \partial_tS_t(\mu)=\int_X\partial_tf_t(x)d\mu(x) \text{ for a.e. }t\in[0,T].
  \end{equation*}
  Moreover the set of differentiability points of $t\mapsto S_t(\mu)$ can be chosen independent of $\mu$.
\end{enumerate}
\end{lma}
\begin{proof}
 The domain is time-independent by virtue of \eqref{eqintimeandspace}. Since $m(X)=1$ we can estimate
$\Ent(\mu)\geq0$ and hence for $\mu\in\mathcal{P}^{ac}(X)$
\begin{equation*}
 S_t(\mu)\geq \int f_t(x)d\mu(x)\geq -||f_t||_{L^\infty}\geq -||f_s||_{L^\infty}-L^*T.
\end{equation*}
If $\mu\notin\mathcal{P}^{ac}(X)$ we know that $S_t(\mu)=\infty$ and we conclude
$\inf_{t,\mu}S_t(\mu)>-\infty$. 

For every $t$ the measure $m_t(X)$ is finite and thus $\mu\to S_t(\mu)$ is lower 
semicontinuous with respect to weak convergence (Lemma 4.1 in \cite{sturm2006}). 

Fix $\mu\in \Dom(S)$. The Lipschitz continuity of $t\mapsto f_t(x)$ ensures $|\partial_tS_t(\mu)|\leq C$. 
It is clear that for every $x\in X$ the map $t\mapsto f_t(x)$
is differentiable for a.e. $t\in[0,T]$. Hence the integral $\int_X\int_{t_1}^{t_2}|\partial_t f_t(x)|dtd\mu(x)$ exists and the Fubini-Tonelli theorem states
\begin{equation*}
 \int_X\int_{t_1}^{t_2}\partial_t f_t(x)dtd\mu(x)=\int_{t_1}^{t_2}\int_X\partial_t f_t(x)d\mu(x)dt.
\end{equation*}
The Fubini theorem again yields that for a.e. $t$ the map $x\mapsto \partial_tf_t(x)$ is $\mu$-integrable and so for a.e. $t$ the integral
\begin{equation*}
\int_X\partial_tf_t(x)d\mu(x)
\end{equation*}
exists. Take a differentiability point $t$ of $f_t(x)$. Then for $\mu$-a.e. $x\in X$
\begin{equation*}
 \lim_{h\to 0}\frac1h(f_{t+h}-f_t)(x)=\partial_t f_t(x), \text{ and }\frac1h|(f_{t+h}-f_t)(x)|\leq L^*.
\end{equation*}
Hence we conclude that for a.e. $t\in[0,T]$
\begin{align*}
 \lim_{h\to0}\frac1h[S_{t+h}(\mu)-S_t(\mu)]=\lim_{h\to0}\int\frac1h[f_{t+h}(x)-f_t(x)]d\mu(x)=\int \partial_tf_t(x)d\mu(x),
\end{align*}
where the last equality is due to the dominated convergence theorem. Finally, for $\mu\ll m$, the inclusions
\begin{align*}
 &\left\{t\in[0,T]\big| \lim\frac1h[S_{t+h}(\mu)-S_t(\mu)] \text{ exists}\right\}\\
 \subset &\left\{t\in[0,T]\big| \lim\frac1h[f_{t+h}(x)-f_t(x)] \text{ exists for }\mu \text{ a.e. }x\right\}\\
 \subset& \left\{t\in[0,T]\big| \lim\frac1h[f_{t+h}(x)-f_t(x)] \text{ exists for }m \text{ a.e. }x\right\}
\end{align*}
show that the set of differentiability points of $t\mapsto S_t(\mu)$ does not depend on $\mu$, since the complement
$\left\{t\in[0,T]\big| \lim\frac1h[f_{t+h}(x)-f_t(x)] \text{ exists for }m \text{ a.e. }x\right\}^C$ is negligible.
\end{proof}

Since we want to apply the results from Section \ref{sec:gradflow}, we still need to check the assumptions in Proposition \ref{stability}. It has been shown in \cite{gigli2009} that if the Ricci curvature of $(X,d_t,m_t)$ is bounded from below by $K\in \mathbb R$ the squared slope of the entropy is lower semicontinuous.
We briefly recall the arguments.
\begin{Def}\label{goodplan}
The set $\GP\subset\mathcal{P}(X^2)$ is the set of plans $\gamma$ such that
\begin{enumerate}
\item the marginals $\pi^i_\#\gamma$, $i=1,2$ are absolutely continuous with densities bounded away from 0 and $\infty$,
\item $\underset{(x,y)\in \supp(\gamma)}\sup d_t(x,y)<\infty$ for some $t\in[0,T]$, and thus for any.
\end{enumerate}
\end{Def}
Given $\gamma\in\GP$ and $\mu\in\mathcal{P}_2^{ac}(X)$, we define the plan $\gamma_\mu\in\mathcal{P}(X^2)$ and the measure $\nu_{\gamma,\mu}\in\mathcal(X)$ as
\begin{equation*}
d\gamma_\mu(x,y)=\frac{d\mu(x)}{d\pi_\#^1\gamma(x)}d\gamma(x,y), \quad \nu_{\gamma,\mu}=\pi_\#^2\gamma_\mu.
\end{equation*}
Note that since $\gamma_\mu\ll\gamma$, we have $\nu_{\gamma,\mu}\ll m$ with density
\begin{equation*}
 g_{\gamma,\mu}(y)=\frac{d\pi_\#^2\gamma(y)}{dm(y)}\int\frac{d\mu(x)}{d\pi_\#^1\gamma(x)}d\gamma_y(x),
\end{equation*}
where $(\gamma_y)_y\subset\mathcal{P}(X)$ is the disintegration of $\gamma$ with respect to its second marginal.

Observe that from 2. of the definition of the set $\GP$ we have that the cost $C_t(\gamma)$ of a plan $\gamma\in\GP$ is always finite and 
$\nu_{\gamma,\mu}\in\mathcal{P}_2(X)$ since $\mu\in\mathcal{P}_2(X)$.

The next Proposition gives an alternative representation formula for the slope in terms of good plans, cf. \cite[Theorem 12]{gigli2009}.
\begin{Prop}
 For every $t\in[0,T]$ and every $\mu\in \Dom(S)$ it holds 
 \begin{equation}
 \begin{aligned}\label{repslope}
  \underset{\nu\neq\mu}{\underset{\nu\in\mathcal{P}_2(X)}\sup}\frac{(S_t(\mu)-S_t(\nu)-\frac{K^-}2W_t^2(\mu,\nu))^+}{W_t(\mu,\nu)}\\
  =\underset{\gamma\in GP}\sup\frac{(S_t(\mu)-S_t(\nu_{\gamma,\mu})-\frac{K^-}2C_t(\gamma_\mu))^+}{\sqrt{C_t(\gamma_\mu)}},
 \end{aligned}
 \end{equation}
where the value of the second expression is taken by definition as 0 if $C_t(\gamma_\mu)=0$.
\end{Prop}
\begin{proof}
We start with proving $\geq$. For this fix a plan $\gamma\in\GP$ such that $\nu_{\gamma,\mu}\neq\mu$. From $C_t(\gamma_\mu)\geq W_t^2(\mu,\nu_{\gamma,\mu})>0$ 
we obtain 
\begin{align*}
 \frac{(S_t(\mu)-S_t(\nu_{\gamma,\mu})-\frac{K^-}2W_t^2(\mu,\nu_{\gamma,\mu}))^+}{W_t(\mu,\nu_{\gamma,\mu})}\\
 \geq \frac{(S_t(\mu)-S_t(\nu_{\gamma,\mu})-\frac{K^-}2C_t(\gamma_\mu))^+}{\sqrt{C_t(\gamma_\mu)}}.
\end{align*}

To show the reverse inequality take $\nu\in\mathcal{P}_2^{ac}(X)$ different from $\mu$. Lemma 10 in \cite{gigli2009} provides
a sequence $(\gamma^n)\subset\GP$
such that $S_t(\nu_{\gamma^n,\mu})\to S_t(\nu)$ and $C_t(\gamma_\mu^n)\to W_t^2(\mu,\nu)$ as $n\to\infty$ and hence
\begin{align*}
 \frac{(S_t(\mu)-S_t(\nu)-\frac{C}2W_t^2(\mu,\nu))^+}{W_t(\mu,\nu)}\\
 =\lim_{n\to\infty} \frac{(S_t(\mu)-S_t(\nu_{\gamma^n,\mu})-\frac{C}2C_t(\gamma^n_\mu))^+}{\sqrt{C_t(\gamma^n_\mu)}},
\end{align*}
which shows $\leq$.

\end{proof}
We get the following as consequence of formula \eqref{repslope}, cf. \cite[Corollary 13]{gigli2009}.

\begin{cor}\label{lscsquaredslope}
 Suppose that $S$ is $K$-convex. Then for every $t\in[0,T]$
 \begin{equation*}
  |\nabla_tS_t|^2(\mu)\leq\liminf |\nabla_{t}S_t|^2(\mu_n),
 \end{equation*}
 whenever $\mu_n\rightharpoonup\mu$ as $n\to\infty$ such that $\sup_nS_t(\mu_n)<\infty$. Further $\mu\mapsto |\nabla_tS_t|^2$ is convex with respect to linear 
 interpolation on the sublevels of $S$.

\end{cor}
\begin{proof}
 Consider the map $\mu\mapsto C_t(\gamma_\mu)$. It is clearly linear. Also, one can show that it is weakly continuous on sublevels of the entropy.
 From \cite[Proposition 11]{gigli2009} we further know that 
 $\mu\mapsto S_t(\mu)-S_t(\nu_{\gamma,\mu})$ is lower semicontinuous with respect to weak convergence on sublevels of the entropy and convex with respect
 to linear interpolation. Hence 
 \begin{equation*}
  \mu\mapsto S_t(\mu)-S_t(\nu_{\gamma,\nu})-\frac{K^-}2C_t(\gamma_\mu)
 \end{equation*}
 is lower semicontinuous with respect to weak convergence on the sublevels of the entropy. The same holds true for its positive part.
 Now apply that the function $\varPsi\colon \mathbb{R}^2\to\mathbb{R}$ defined by
 \begin{equation*}
   \varPsi(a,b) =
   \begin{cases}
     \frac{a^{2}}{b} & \text{if } b > 0, \\
     0 & \text{if } a=b=0, \\
     +\infty  & \text{if } a\neq 0,b=0 \textit{ or } b<0,
   \end{cases}
\end{equation*}
is convex, continuous on $[0,\infty)^2\setminus\{(0,0)\}$ and increasing in $a$, and the conclusion follows. From formula \eqref{repslope} the assertion follows.

\end{proof}

\subsection{Existence and Uniqueness of EDE-gradient flow for the entropy}
In this section we want to show existence and uniqueness of the dynamic EDI-gradient flow with respect to the functional $S$ 
on the complete geodesic space $(\mathcal{P}_2(X),W_t)$. For this we additionally
have to assume that $X$ is boundedly compact, i.e. closed balls are compact. For this reason we can take \textbf{A5} for granted, as shown in the next lemma.

\begin{lma}
 Assume that $X$ is boundedly compact.
 Then the following holds true.
If $(\mu_n)\subset \mathcal{P}_2(X)$ with $\sup_{n,m}W_t(\mu_n,\mu_m)<\infty$, then $(\mu_n)_n$ is sequentially precompact with respect to weak convergence.

\end{lma}

\begin{proof}

If $\sup_{n,m}W_t(\mu_n,\mu_m)<\infty$ for a sequence $(\mu_n)\subset 
\mathcal{P}_2(X)$ the second moments are uniformly bounded. Then Lemma 16 in \cite{gigli2009} implies that $(\mu_n)$ is tight. Applying 
Prokhorov's theorem we infer that $(\mu_n)$ is weakly sequentially precompact.

\end{proof}

\begin{thm}\label{existence}
Assume additionally that $X$ is boundedly compact.
 Suppose that $S$ is $K$-convex for some $K\in\mathbb R$. 
 Then for every $\bar\mu\in \Dom(S)$ there exists a curve $(\mu_t)\in \AC^2([0,T],\mathcal{P}_2(X))$ starting in $\bar\mu$ and satisfying
\begin{equation}\label{eq:steepestdescent}
  S_t(\mu_t)+\frac12\int_0^t|\dot\mu_r|^2_rdr+\frac12\int_0^t|\nabla_rS_r|^2(\mu_r)dr
  \leq S_0(\mu_0)+\int_0^t(\partial_r S_r)(\mu_r)dr,
 \end{equation}
 for every $t\in[0,T]$.
\end{thm}
\begin{proof}
We may apply Proposition \ref{compactness} and obtain a limit curve $\mu\in \AC^2([0,T];\mathcal P_2(X))$ starting in $\bar\mu$ such that
\begin{equation*}
 \bar\mu_t^{h_n}\rightharpoonup\mu_t, \text{ and } \tilde\mu_t^{h_n}\rightharpoonup\mu_t \quad \forall t\in[0,T],
\end{equation*}
where $\bar\mu^h$ and $\tilde\mu^h$ are defined as in Definition \ref{interpolants}.
and satisfy by Proposition \ref{discreteede}
 \begin{align*}
S_{\overline{h_n}(t)}(\bar \mu_{t}^{h_n})+\frac12\int_{0}^{\overline{h_n}(t)}(\Dsp_r^{h_n})^2dr+\frac12\int_{0}^{\overline{h_n}(t)}(\Dsl_r^{h_n})^2dr\\
=S_{0}(\bar \mu)+\int_{0}^{\overline{h_n}(t)}(\partial_rS_r)(\tilde \mu_r^{h_n})dr.
\end{align*}
From Corollary \ref{lscsquaredslope}, Lemma \ref{slopeest} and Lemma \ref{modest} together with \eqref{loglip}, applying Fatou's Lemma we obtain
\begin{align*}
 &\int_0^t|\nabla_rS_r|^2(\mu_r)dr\leq\liminf_{n\to\infty}\int_0^t|\nabla_rS_r|^2(\tilde\mu_r^{h_n})dr\\
 &\leq \liminf_{n\to\infty}\int_0^t(|\nabla_{\overline{h_n}(r)}S_r|(\tilde\mu_r^{h_n})+|\nabla_rS_r|(\tilde\mu_r^{h_n})-|\nabla_{\overline{h_n}(r)}S_r|(\tilde\mu_r^{h_n}))^2dr\\
 &\leq \liminf_{n\to\infty} \bigg[\int_0^t(\Dsl_r^{h_n})^2dr+2Ch_n\int(\Dsl_r^{h_n})^2dr+Ch_n^2\int(\Dsl_r^{h_n})^2dr\bigg].
\end{align*}
We deduce
\begin{equation*}
 \int_0^t|\nabla_rS_r|^2(\mu_r)dr\leq\liminf\int_0^t(\Dsl_r^{h_n})^2dr
\end{equation*}
from the estimate $\int_0^t(\Dsl_r^{h_n})^2dr\leq S_0(\mu)+L^*T-\inf_{t,\mu}S_t(\mu)$.

To show that \eqref{eq:steepestdescent} is valid, it is left to show that
\begin{equation*}
 \lim_{n\to\infty}\int_0^t(\partial_rS_r)(\tilde\mu_r^{h_n})dr=\int_0^t(\partial_rS_r)(\mu_r)\, dr.
\end{equation*}
This already follows if
we prove that a stronger convergence than weak convergence of measures holds true. In fact, from 
\eqref{eq1:apest} we know that there exists a density $\tilde\rho_t^{h_n}=d\tilde\mu_t^{h_n}/dm\in L^1(X,m)$
for every $t\in[0,T]$ and $n\in
\mathbb{N}$. The lower semicontinuity of the entropy implies that $\sup_tS_t(\mu_t)<\infty$ and thus $\mu_t=\rho_tm$, for some $\rho_t\in L^1(X,m)$.
Choose an arbitrary subsequence $h_{n_k}$. Then since the family of densities $(\tilde\rho_t^{h_{n_k}})_k$ is equiintegrable, i.e.
\begin{equation*}
 \sup_k\int_X\max\{0,\rho_t^{h_{n_k}}\log\rho_t^{h_{n_k}}\}dm<\infty,
\end{equation*}
(cf. \cite[Theorem 4.5.9]{bogachev2007}), the Dunford-Pettis Theorem (\cite[Corollary 4.7.19]{bogachev2007}) ensures
that there exists a subsubsequence $\tilde\rho_{t}^{h_{n_{k_l}}}$ that converges in the weak
topology of $L^1(X,m)$ to the function $\rho_t\in L^1(X,m)$. Hence for the original subsequence we already have
\begin{equation*}
 \tilde\rho_t^{h_n}\rightharpoonup\rho_t \text{ in }L^1(X,m) \quad\forall t\in[0,T].
\end{equation*}
As a direct consequence we obtain \eqref{eq:steepestdescent}, since similar as in Proposition \ref{stability}
\begin{align*}
  &S_t(\mu_t)+\frac12\int_0^t|\dot \mu|^2_rdr+\frac12\int_0^t|\nabla_rS_r|^2(\mu_r)dr\\
 \leq &\liminf_{n\to\infty} \left[S_{\overline{h_n}(t)}(\bar \mu_{t}^{h_n})+\frac12\int_0^{\overline{h_n}(t)} (Dsp_r^{h_n})^2dr 
 +\frac12\int_0^{\overline{h_n}(t)}(\Dsl_r^{h_n})^2dr\right]\\
 \leq & \liminf_{n\to\infty}\left[S_0(\bar \mu)+\int_0^{\overline{h_n}(t)}(\partial_rS_r)(\tilde \mu_r^{h_n})dr\right]\\
 \leq& S_0(\bar \mu)+\int_0^t(\partial_r S_r)(\mu_r)dr.
\end{align*}

\end{proof}

\begin{remark}
Actually, the statement of Theorem \ref{existence} holds true without assuming that $X$ is boundedly compact, since $m$ is assumed to be finite.
 If $(X,d)$ is a Polish space and $m\in \mathcal{P}(X)$ we may apply $z\log z\geq-1/e$
 and Jensen's inequality to obtain
 \begin{equation*}
  \Ent_m(\mu)\geq\mu(E)\log\left(\frac{\mu(E)}{m(E)}\right)-\frac1e \quad\forall E\in \mathcal{B}(X).
 \end{equation*}
Taking into account that the singleton $\{m\}$ is tight, this shows tightness of the sublevels of the entropy since $\mu(E)\to0$ as $m(E)\to 0$. 
Hence we could replace our assumption \textbf{A5} in section \ref{sec:gradflow} by assuming that for each $t\in[0,T]$ the sublevels of the functional 
are sequentially $\sigma$-compact.
See also \cite[Remark 7.3]{agscalc}.

\end{remark}

\begin{thm}\label{uniqueness}
 Assume $S$ is $K$-convex and $\bar\mu\in \Dom(S)$. Then there exists at most one dynamic EDI-gradient flow. Moreover we have equality in 
 \eqref{eq:steepestdescent}, i.e. for every $t\in[0,T]$ it holds the following dynamic EDE
 \begin{equation*}
  S_t(\mu_t)+\frac12\int_0^t|\dot\mu_r|^2_rdr+\frac12\int_0^t|\nabla_rS_r|^2(\mu_r)dr
  = S_0(\mu_0)+\int_0^t(\partial_r S_r)(\mu_r)dr.
 \end{equation*}
 
\end{thm}
\begin{proof}
 Let us first observe that a weak chain rule for gradient flows is applicable. For this we prove that a variant of the assumption in Proposition 
 \ref{weakchainrule} concerning the time derivative is satisfied 
 by the entropy $S_t$. We choose a sequence $\mu_n=\rho_nm$ converging to $\mu=\rho m$ such that $\sup_nS(\mu_n)<\infty$. We need to show that
 for almost every $t$ 
 \begin{equation}\label{entropyweakchainrule}
 \lim_{n\to\infty}\frac{S_{t_n}(\mu_n)-S_t(\mu_n)}{t_n-t}=\lim_{n\to\infty}\int_X\frac{f_{t_n}(x)-f_t(x)}{t_n-t}\rho_n(x)dm(x)=\partial_tS_t(\mu),
 \end{equation}
if $t_n\searrow t$ as $n\to\infty$. This would imply the weak chain rule \eqref{eq:weakchainrule} in Proposition \ref{weakchainrule} 
restricted to curves which are contained in the sublevels of the functional. In order to show \eqref{entropyweakchainrule} note that as in the proof of Theorem 
 \ref{existence} the sequence $(\rho_n)$ is equi-integrable and thus $\rho_n$ converges to $\rho$ in duality with $L^\infty$ functions. Then we decompose
 \begin{align*}
& \int_X\frac{f_{t_n}-f_t}{t_n-t}\rho_ndm=\int_X\Big(\frac{f_{t_n}-f_t}{t_n-t}-\partial_tf_t\Big)\rho_ndm+\int \partial_t f_t\rho_n\, dm\\
& =\int_{|\rho_n|<M}\Big(\frac{f_{t_n}-f_t}{t_n-t}-\partial_tf_t\Big)\rho_ndm+\int_{|\rho_n|\geq M}\Big(\frac{f_{t_n}-f_t}{t_n-t}-\partial_tf_t\Big)\rho_ndm\\
& \quad+\int \partial_t f_t\rho_n\, dm.
 \end{align*}
 The third integral clearly converges to $\int \partial_tf_t\rho\, dm=\partial_tS_t(\mu)$ by Lemma \ref{ent}, while the first vanishes by dominated convergence. The second vanishes after letting $n\to\infty$ and then $M\to\infty$ by equi-integrability of $(\rho_n)$.

 Let us assume that there exist two dynamic EDI-gradient flows $(\mu_t^1)$, $(\mu_t^2)$ starting from $\bar\mu\in \Dom(S)$. As seen in the proof of Theorem \ref{existence}
 we know that these curves are contained in the sublevels of $S$ and hence together with the weak chain rule it follows
\begin{equation*}
   S_0(\bar\mu)=S_t(\mu_t^1)+\frac12\int_0^t|\dot\mu^1|^2_rdr+\frac12\int_0^t|\nabla_rS_r|^2(\mu_r^1)dr
 -\int_0^t\partial_r S_r(\mu_r^1)dr,
 \end{equation*}
 \begin{equation*}
   S_0(\bar\mu)=S_t(\mu_t^2)+\frac12\int_0^t|\dot\mu^2|^2_rdr+\frac12\int_0^t|\nabla_rS_r|^2(\mu_r^2)dr
 -\int_0^t\partial_r S_r(\mu_r^2)dr.
 \end{equation*}
Now define
\begin{equation*}
 \mu_t^{1,2}=\frac{\mu_t^1+\mu_t^2}2 \quad t\geq0.
\end{equation*}
Then $\mu_0^{1,2}=\bar\mu$ and from the strict convexity of the entropy, the convexity of the squared slope (Corollary \ref{lscsquaredslope}), 
the convexity of the squared speed  (Lemma \ref{cofsquaredslope}) and the linearity of $\partial _r S_r$ (Lemma \ref{ent}) we have that
\begin{equation*}
   S_0(\bar\mu)>S_t(\mu_t^{1,2})+\frac12\int_0^t|\dot\mu^{1,2}|^2_rdr+\frac12\int_0^t|\nabla_rS_r|^2(\mu_r^{1,2})dr
 -\int_0^t\partial_r S_r(\mu_r^{1,2})dr,
 \end{equation*}
 whenever these curves are different. But since \eqref{entropyweakchainrule} is applicable to $\mu_t^{1,2}$,
 this contradicts \eqref{eq:weakchainrule}.
\end{proof}

\section{Dynamic gradient flows in Hilbert spaces}\label{sechilbert}
Let $H$ be a separable Hilbert space with a family of scalar products $(\langle\cdot,\cdot\rangle_t)$. We assume that \eqref{loglip} holds for the distances
$||x-y||_t:=\sqrt{\langle x-y,x-y\rangle_t}$. Let $E\colon[0,T]\times H\to \mathbb R\cup\{+\infty\}$ be a functional such that
$x\mapsto E_t(x)$ is convex and lower semicontinuous. Again we require that the domain $\Dom(E_t)=\{x\colon E_t(x)<\infty\}$ is time independent.
The subdifferential $D^-_t E_t(x)$ of $E_t$ at some $x\in \Dom(E)$ is the set of all $v\in H$ such that 
\begin{equation*}
 E_t(y)-E_t(x)\geq \langle v,y-x\rangle_t\qquad \forall y\in H.
\end{equation*}
It follows from the definition of the subdifferential that $D^-_t E_t$ is \emph{monotone}, i.e. for every $v\in D_t^-E_t(x)$, 
$w\in D_t^-E_t(y)$ we have
\begin{align}\label{monotone}
 \langle v-w,x-y\rangle_t\geq0.
\end{align}

Note that $D_t^-E_t(x)$ is closed and convex. Hence we can set $\nabla_tE_t(x)$ as the element of minimal $||\cdot||_t$-norm in $D^-_tE_t(x)$ as soon as
$D_t^-E_t(x)\neq\emptyset$.

\begin{Def}\label{defgradflowhilbert}
 We say that $(x_t)$ is a \emph{dynamic gradient flow} for $E_t$ starting from $x\in H$ if it is locally absolutely continuous and
 \begin{equation*}
  \partial_t x_t\in-D_t^- E_t(x_t) \text{ for a.e. } t>0
 \end{equation*}
 and $\lim_{t\searrow 0} x_t=x$.

\end{Def}

We cannot hope to have a minimal selection result, i.e. $\frac{d^+}{dt}x_t=-\nabla_t E_t(x_t)$. We illustrate this in the following example.
\begin{ex}\label{minsel}
 Consider once again the energy functional $E_t(x)=|x-t|$ on $\mathbb R$. Then the curve $x_t=t$ defines a gradient flow of 
 $E_t$ since 
 $\partial_t x_t=1\in-D^-E_t(x_t)$, but $\partial_t x_t\neq -\nabla E_t(x_t)=0$.
\end{ex}

In the following we show that the gradient flow in the sense of Definition \ref{defgradflowhilbert} is a dynamic forward EVI$(-L/2,\infty)$ gradient flow
introduced in Section \ref{secevi}.
We recall that for $s,t \in[0,T]$, $\gamma\in\AC^2([0,1];H)$ the action of the curve 
\begin{align*}
 \mathcal A_{s,t}(\gamma):=\lim_{h\to0}\sup\Big\{\sum_{i=1}^n(a_{i}-a_{i-1})^{-1}||\gamma_{a_i}-\gamma_{a_{i-1}}||_{s+a(t-s)}^2\Big\},
\end{align*}
where the supremum runs over all partitions $0=a_0<a_1<\cdots a_n=1$ such that $a_i-a_{i-1}\leq h$ for some $h>0$. 

For $x,y\in H$ we define
\begin{align*}
 ||x-y||_{s,t}^2:=\inf A_{s,t}(\gamma),
\end{align*}
where the infimum runs over all curves $\gamma\in\AC^2([0,1];H)$ such that $\gamma_0=x$ and $\gamma_1=y$. 

\begin{Prop}\label{blabla}
 Let $E\colon[0,T]\times H\to(-\infty,+\infty]$ be a functional such that $x\mapsto E_t(x)$ is convex and lower semicontinuous for each $t\in[0,T]$.
 Let $(x_t)$ be a gradient flow of $E$ in the sense of Definition \ref{defgradflowhilbert}.
 Then, with $L$ denoting the the logarithmic Lipschitz control \eqref{loglip} of the distances, $(x_t)$ is a dynamic forward EVI$(-L/2,\infty)$ gradient flow,
 i.e. for all $y\in \Dom(E)$ and a.e. $t$
 \begin{align*}
  \frac12\partial^+_s||x_s-y||^2_{s,t}\Big|_{s=t}-\frac{L}4||x_t-y||_t^2\leq E_t(y)-E_t(x_t).
 \end{align*}
\end{Prop}
\begin{proof}
Let $y\in \Dom(E)$. Then
\begin{align*}
 \frac12\partial_s^+||x_s-y||^2_{s,t}\Big|_{s=t}=&\limsup_{s\to t}\left\{\frac1{2(s-t)}(||x_s-y||^2_{s,t}-||x_t-y||^2_t)\right\}\\
 \leq&\limsup_{s\to t}\left\{\frac1{2(s-t)}(||x_s-y||^2_{s,t}-||x_s-y||^2_t)\right\}\\
 +&\limsup_{s\to t}\left\{\frac1{2(s-t)}(||x_s-y||^2_t-||x_t-y||^2_t)\right\}
 \end{align*}

The first limsup can be estimated with the help of Proposition 7.2(iii) in \cite{sturm2016} by
\begin{align*}
 &\limsup_{s\to t}\left\{\frac1{2(s-t)}(||x_s-y||^2_{s,t}-||x_s-y||^2_t)\right\}\\
 \leq  &\limsup_{s\to t}\left\{\frac{1}{2(s-t)}\Big(\frac{e^{L|t-s|}-1}{L|t-s|}-1\Big)||x_s-y||^2_{s}\right\}\\
 =&\limsup_{s\to t}\left\{\frac{1}{2(s-t)}\Big(\frac{\frac12L|t-s|^2+o(|t-s|^2)}{L|t-s|}\Big)||x_s-y||^2_{s}\right\}\\
 \leq &\frac{L}4||x_t-y||^2_t,
\end{align*}
where the last inequality follows from the continuity of $t\mapsto x_t$ and $t\mapsto ||\cdot||_t$.
For the second limsup we apply that $(x_t)$ is supposed to be a gradient flow of $E$;
\begin{align*}
 &\limsup_{s\to t}\left\{\frac1{2(s-t)}(||x_s-y||^2_t-||x_t-y||^2_t)\right\}\\
 =&\langle x_t-y,\partial_t x_t\rangle_t\leq E_t(y)-E_t(x_t)
\end{align*}
for a.e. $t\geq 0$. Combining these two observations we conclude
\begin{align*}
  \frac12\partial_s^+||x_s-y||^2_{s,t}\Big|_{s=t}\leq\frac{L}4||x_t-y||^2_t+E_t(y)-E_t(x_t),
\end{align*}
which proves the claim.
\end{proof}

\subsection{Existence and Uniqueness}

We assume that the following holds for the energy functional, cf. \cite{rossi}.
\begin{enumerate}
 \item $x\mapsto E_t(x)$ is lower semicontinuous\\
 and $E_t(x)\geq 0 \quad\forall (t,x)\in[0,T]\times \Dom(E)$,
 \item $\exists\, C_1\, \forall x\in\Dom(E)\, \forall s,t\in[0,T]:\quad |E_t(x)-E_s(x)|\leq C_1E_t(x)|t-s|$.
\end{enumerate}
By virtue of the functional's lower semicontinuity we obtain that if
$v_n\in D_{t}^-E_t(x_n)$ and $x_n\to x$, $v_n\rightharpoonup v$, then $v\in D_t^-E_t(x)$.

We write 
\begin{equation*}
 e(x):=\sup_{t\in[0,T]}E_t(x).
\end{equation*}
Note that from the Lipschitz property it follows that there exists a constant $C_2>0$ such that for all $x\in \Dom(E)$
\begin{equation}\label{infsup}
 e(x)\leq C_2 \inf_{t\in[0,T]} E_t(x).
\end{equation}

\subsubsection*{Approximation}
We fix a time step $h>0$ and subdivide the interval $[0,T]$ into
\begin{align*}
 \mathcal{P}_h:=\{t_0=0<t_1<\cdots<t_{N-1}<T\leq t_N\},\qquad t_n=nh, N\in\mathbb{N}.
\end{align*}
For $0\leq t\leq T$ we define the piecewise constant interpolants $\overline{h}(t)$ and $\underline{h}(t)$ associated with the partition $\mathcal{P}_h$ 
in the following way;
\begin{equation}\label{dicretetime1}
\overline{h}(0)=0=\underline{h}(0), \text{ and for }t\in (t_{n-1},t_n] \quad \overline{h}(t)=t_n , \quad \underline{h}(t)=t_{n-1}.
\end{equation}
The definition implies that $\overline{h}(t)\searrow t$ and $\underline{h}(t)\nearrow t$ if $h\searrow0$.

For a given initial value $\bar x$ we recursively define a sequence $(x_n^h)$ of minimizers by
\begin{align}\label{minprob1}
 x_0^h:=\bar x, \qquad x_n^h:=\arg\min_{x}\Big\{E_{t_n}(x)+\frac1{2h}||x-x_{n-1}^h||_{t_n}^2\Big\}.
\end{align}
We can argue as in the proof of Proposition \ref{minexist} and directly obtain 
for every $\bar x\in \Dom(E)$ and $h>0$ a (unique) solution to the minimization problem \eqref{minprob1}.
As in Section \ref{sec:gradflow} we define piecewise constant interpolants by setting
\begin{equation*}
 \bar x_t^h:=x_n^h \text{ for }t\in(t_{n-1},t_n],      \qquad \underline{x}_t^h:=x_{n-1}^h\text{ for }t\in(t_{n-1},t_n],
\end{equation*}
and moreover, the piecewise linear interpolant
\begin{equation*}
 x_t^h=\frac{t-t_{n-1}}{h}x_n^h+\frac{t_n-t}h x_{n-1}^h \text{ for }t\in[t_{n-1},t_n).
\end{equation*}
For $t\in(t_{n-1},t_n)$ we denote the time derivative of $t\mapsto x_t^h$ by $\dot x_t^h$.

Recall that the variational interpolation is a map $t\to\tilde x_t^h$ defined by
\begin{align*}
\tilde x_t^h=\arg\min_x\Big\{E_{t}(x)+\frac1{2r}||x-x_{n-1}^h||^2_{t_n}\Big\},\\
\text{for }t=t_{n-1}+r\in(t_{n-1},t_n],
\end{align*}
and $\tilde x_0^h=\bar x$. Finally we define $t\mapsto\tilde v_t^h$ by
\begin{equation*}
 \tilde v_t^h:=\frac{\tilde x_t^h-x_{n-1}^h}{t-t_{n-1}}\quad \forall t\in(t_{n-1},t_n].
\end{equation*}

As in Section \ref{sec:gradflow}, in order to extract a converging subsequence, we proof a priori estimates on the discrete solutions. The proof is along the lines 
of Proposition 6.3 in \cite{rms}.
\begin{Prop}
 The following inequality holds for the interpolants $\bar x^h$, $x^h$, $\tilde x^h$ and $\tilde v^h$
 \begin{equation}
 \begin{aligned}\label{eqdiscede}
 E_{\overline{h}(t)}(\bar x_{t}^h)+\frac12\int_{\overline{h}(s)}^{\overline{h}(t)}||\dot x_r^h||^2_{\overline h(r)}dr+
 \frac12\int_{\overline{h}(s)}^{\overline{h}(t)}||\tilde v_r^h||^2_{\overline h(r)}dr\\
\leq E_{\overline{h}(s)}(\bar x_{s}^h)+C_1\int_{\overline{h}(s)}^{\overline{h}(t)}e(\tilde x_r^h)dr.
 \end{aligned}
 \end{equation}
In particular there exists a constant $M$ such that for all $h>0$
\begin{equation}\label{est}
\sup_{t\in(0,T)}e(\bar x_t^h)\leq M, \qquad \sum_{n=0}^N\frac1h||x_n^h-x_{n-1}^h||^2_{t_n}\leq M,
\end{equation}
\begin{equation}\label{est2}
 \int_0^T||\dot x_r^h||^2_{\overline h(r)}dr\leq M,\qquad \int_0^T||\tilde v_r^h||^2_{\overline h(r)}dr\leq M.
\end{equation}
Moreover
\begin{align}\label{est3}
 ||\tilde x_t^h-\underline x_t^h||^2\in O(h),\quad ||x_t^h-\bar x_t^h||^2\in O(h), \quad ||\bar x_t^h-\underline x_t^h||^2\in O(h).
\end{align}

\end{Prop}
\begin{proof}
 Consider the map
 \begin{equation*}
  r\mapsto J_{s,r}(y):=\min_x\Big\{E_{s+r}(x)+\frac1{2r}||x-y||^2_{s+h}\Big\}
 \end{equation*}
for a given $s\in[0,T]$, $y\in D$, $0<r<T-s$. We claim that this map is differentiable almost everywhere in $(0,T-s)$ and for every $r_0\in(0,T-s)$
for the minimizer $(0,r_0]\ni r\mapsto x_r$ it holds
\begin{equation}\label{int}
\begin{aligned}
 \frac1{2r_0}||x_{r_0}-y||^2_{s+h}+\int_0^{r_0}\frac1{2r^2}||x_r-y||^2_{s+h}dr+ E_{s+r_0}(x_{r_0})\\
 \leq E_s(y)+C_1\int_0^{r_0}e(x_r)dr.
\end{aligned}
\end{equation}
Indeed, arguing similar as in \eqref{jlipschitz} we obtain that for $r_1<r_2\in(0,T-s)$
\begin{equation}
\begin{aligned}\label{jlipschitz2}
 &J_{s,r_2}(y)-J_{s,r_1}(y)-(E_{s+r_2}(x_{r_1})-E_{s+r_1}(x_{r_1}))\\
 &\leq -\frac1{2r_1r_2}(r_2-r_1)||x_{r_1}-y||^2_{s+h}\leq 0,
\end{aligned}
\end{equation}
hence the map $r\mapsto J_{s,r}(y)$ is the sum of a locally Lipschitz and of a nonincreasing function
\begin{equation*}
  J_{s,r_2}(y)\leq J_{s,r_1}(y)+(r_2-r_1)C_1 e(x_{r_1}),
\end{equation*}
and differentiable almost everywhere. So let $r\in(0,T-s)$ be a differentiable point of $r\mapsto J_{s,r}(y)$. Then with \eqref{jlipschitz2} we get
\begin{align*}
 &\frac{d}{dr} J_{s,r}(y)+\frac1{2r^2}||x_r-y||^2_{s+h}\\
 =&\lim_{h\to0}\Big(\frac{J_{s,r+h}(y)-J_{s,r}(y)}{h}+\frac1{2(r+h)r}||x_r-y||^2_{s+h}\Big)\\
 \leq &\liminf_{h\to0}\frac{E_{s+r+h}(x_r)-E_{s+r}(x_r)}{h}\leq C_1 e(x_r),
\end{align*}
and integrating from 0 to $r_0$ gives us \eqref{int}.\\

Applying \eqref{int} with $t\in(t_{n-1},t_n]$, $y=x_{n-1}^h$, $s=t_{n-1}$ and $r_0=t-s$ we obtain for $\tilde x_t^h$
\begin{equation}
\begin{aligned}\label{int3}
 \frac1{2(t-t_{n-1})}&||\tilde x_t^h-x_{n-1}^h||^2_{t_n}+\int_{t_{n-1}}^t\frac1{2(r-t_{n-1})^2}||\tilde x_r^h-x_{n-1}^h||^2_{t_n}dr+ E_{t}(\tilde x_{t}^h)\\
 \leq& E_{t_{n-1}}(x_{n-1}^h)+C_1\int_{t_{n-1}}^{t}e(\tilde x_r^h)dr.
\end{aligned}
\end{equation}
Inserting $t=t_n$ we get for the interpolants $x_t^h$, $\tilde v_t^h$
\begin{equation}
\begin{aligned}\label{int2}
 &\frac12\int_{t_{n-1}}^{t_n}||\dot x_r^h||^2_{t_n}dr+\int_{t_{n-1}}^{t_n}\frac1{2}||\tilde v_r^h||^2_{t_n}dr+ E_{t_n}(\tilde x_{t_n}^h)\\
& \leq E_{t_{n-1}}(x_{n-1}^h)+C_1\int_{t_{n-1}}^{t_n}e(\tilde x_r^h)dr.
\end{aligned}
\end{equation}
Summing over the partition we end up with \eqref{eqdiscede}.

Note that the minimality and \eqref{infsup} imply that for $r\in(t_{n-1},t_n]$, $t=t_{n-1}+r$
\begin{equation*}
 e(\bar x_{t_{n-1}}^h)\geq E_t(\bar x_{t_{n-1}}^h)\geq\frac1{2h}||\tilde x_r^h-\bar x_{t_{n-1}}^h||^2_{t_n}+E_t(\tilde x_r^h)\geq E_t(\tilde x_r^h)
 \geq\frac1{C_2}e(\tilde x_r^h),
\end{equation*}
and hence with \eqref{int2} we can estimate
\begin{equation*}
E_{t_n}(\bar x_{t_n}^h)\leq E_{t_{n-1}}(x_{n-1}^h)+C_1C_2\int_{t_{n-1}}^{t_n}e(\bar x_{t_{n-1}}^h)dr.
\end{equation*}
Summing over the partitions and applying \eqref{infsup} once more we obtain for some constant $C>0$
\begin{equation*}
e(\bar x_{t_n}^h)\leq C(E_{0}(x_{0}^h)+\int_{0}^{t_n}e(\underline{x}_r^h)dr).
\end{equation*}
We obtain the first inequality in \eqref{est} by applying a discrete Gronwall argument (see e.g. \cite[Lemma 4.5]{rossi}).
It directly follows that the right-hand side of \eqref{eqdiscede} is bounded and \eqref{est2} holds. The second inequality in \eqref{est} is a direct 
consequence of the first estimate in \eqref{est2}.

In order to show the first statement in \eqref{est3} recall that \eqref{int3} together with \eqref{est} implies (with some different constant $M$)
\begin{align*}
 ||\tilde x_t^h-\underline x_t^h||^2\leq 2hM.
\end{align*}
The other two assertions in \eqref{est3} follow from \eqref{est2} via H\"older's inequality
\begin{align*}
 ||x_t^h-x_s^h||\leq \int_s^t||\dot x_r^h||dr\leq \sqrt{M(t-s)}\qquad \forall 0<s<t<T.
\end{align*}

\end{proof}

The following result provides the compactness of the approximate solutions.

\begin{Prop}\label{converge}
For every sequence of time-steps $(h_j)_{j\in\mathbb N}$ such that $h_j\to0$ as $j\to\infty$ there exists a subsequence $h_j$ (not relabeled)
and an absolutely continuous curve $(x_t)\subset \AC^2([0,T];H)$ and  such that 
 \begin{equation*}
 x_t^{h_j}\to x_t \text{ in }\mathcal C^0([0,T];H),
 \end{equation*}
 and
 \begin{equation*}
  \dot x_t^{h_j}\rightharpoonup \dot x_t \text{ in }L^2([0,T];H).
 \end{equation*}
 Moreover for each $t$ $\bar x_t^{h_j}, \tilde x_t^{h_j}\to x_t$ in $H$.
\end{Prop}

 \begin{proof}
 Let $0<g, h<< T$ be two stepsizes and $\{t_n^g\}_{n=0}^{N_g}$, $\{t_n^h\}_{n=0}^{N_h}$ the corresponding partitions of the interval $[0,1]$.
 Let $\{x_n^h\}_{n=0}^{N_h}$ and $\{x_n^g\}_{n=0}^{N_g}$ be the solution to the minimizing problem \eqref{minprob1} with respect to the stepsizes $h$ and
 $g$ respectively with initial condition $x_0^h$ and $x_0^g$. The Euler-Lagrange equation of $x_n^h$ is 
 \begin{equation*}
  \frac{x_n^h-x_{n-1}^h}h\in -D_{t_n^h}^-E_{t_n^h}(x_n^h),
 \end{equation*}
 i.e. 
 \begin{equation*}
  h^{-1}\langle x_n^h-x_{n-1}^h,x_n^h-y\rangle_{t_n^h}+ E_{t_n^h}(x_n^h)-E_{t_n^h}(y)\leq 0 \quad\forall y\in H.
 \end{equation*}
 Inserting the definition of the piecewise linear interpolation $x_t^h$ at $t\in(t_{n-1}^h,t_n^h)$
 \begin{equation}
 \begin{aligned}\label{el1}
  &\langle \dot x_t^h,x_t^h-y\rangle_{t_n^h}+ E_{t_n^h}(x_t^h)-E_{t_n^h}(y)\\
  &\leq (t-t_n^h)\Big(||\dot x_t^h||^2_{t_n^h}+\frac1h(E_{t_n^h}(x_n^h)-E_{t_n^h}(x_{n-1}^h))\Big) \quad\forall y\in H,
 \end{aligned}
 \end{equation}
 where we applied the convexity of $x\mapsto E_t(x)$. The same argumentation for $x_t^g$ at $t\in(t_{m-1}^g,t_m^g)$ yields
 \begin{equation}
 \begin{aligned}\label{el2}
  &\langle \dot x_t^g,x_t^g-y\rangle_{t_m^g}+ E_{t_m^g}(x_t^g)-E_{t_m^g}(y)\\
  &\leq (t-t_m^g)\Big(||\dot x_t^g||^2_{t_m^g}+\frac1g(E_{t_m^g}(x_m^g)-E_{t_m^g}(x_{m-1}^g))\Big) \quad\forall y\in H.
 \end{aligned}
 \end{equation}
 For $t\in(t_{n-1}^h,t_n^h)\cap(t_{m-1}^g,t_m^g)$ we get by putting $y=x_t^g$ into \eqref{el1} and $y=x_t^h$ into \eqref{el2} and adding them
 \begin{equation}
 \begin{aligned}\label{el3}
  &\langle \dot x_t^h,x_t^h-x_t^g\rangle_{t_n^h}+\langle \dot x_t^g,x_t^g-x_t^h\rangle_{t_m^g}\\
  &+ E_{t_n^h}(x_t^h)-E_{t_m^g}(x_t^h)
  + E_{t_m^g}(x_t^g)-E_{t_n^h}(x_t^g)\\
  &\leq (t-t_n^h)\Big(||\dot x_t^h||^2_{t_n^h}+\frac{E_{t_n^h}(x_n^h)-E_{t_n^h}(x_{n-1}^h)}h\Big)\\
  &+(t-t_m^g)\Big(||\dot x_t^g||^2_{t_m^g}+\frac{E_{t_m^g}(x_m^g)-E_{t_m^g}(x_{m-1}^g)}g\Big).
 \end{aligned}
 \end{equation}
 The Lipschitz property \eqref{loglip} of the metric together with the polarization identity gives
 \begin{align*}
  &\langle \dot x_t^h,x_t^h-x_t^g\rangle_{t_n^h}+\langle \dot x_t^g,x_t^g-x_t^h\rangle_{t_m^g}\\
  &\geq \langle \dot x_t^h-\dot x_t^g,x_t^h-x_t^g\rangle_{t_m^g}
  -L|t_n^h-t_m^g| \Big(\langle \dot x_t^h,x_t^h-x_t^g\rangle_{t_m^g}+\frac12||\dot x_t^h-(x_t^h-x_t^g)||^2_{t_m^g}\Big),
 \end{align*}
while the Lipschitz property of the energy yields
\begin{align*}
  E_{t_n^h}(x_t^h)-E_{t_m^g}(x_t^h)
  + E_{t_m^g}(x_t^g)-E_{t_n^h}(x_t^g)\geq -C_1|t_n^h-t_m^g|\Big(E_{t_n^h}(x_t^h)+E_{t_n^h}(x_t^g)\Big).
\end{align*}
Inserting these two inequalities into \eqref{el3} we find
\begin{equation}
 \begin{aligned}\label{el4}
 & \frac{d}{dt}||x_t^h-x_t^g||_{t_m^g}^2=\langle \dot x_t^h-\dot x_t^g,x_t^h-x_t^g\rangle_{t_m^g}\\
  &\leq (t-t_n^h)\Big(||\dot x_t^h||^2_{t_n^h}+\frac{E_{t_n^h}(x_n^h)-E_{t_n^h}(x_{n-1}^h)}h\Big)\\
  &+(t-t_m^g)\Big(||\dot x_t^g||^2_{t_m^g}+\frac{E_{t_m^g}(x_m^g)-E_{t_m^g}(x_{m-1}^g)}g\Big)\\
  &+L|t_n^h-t_m^g| \Big(\langle \dot x_t^h,x_t^h-x_t^g\rangle_{t_m^g}+\frac12||\dot x_t^h-(x_t^h-x_t^g)||^2_{t_m^g}\Big)\\
  &+C_1|t_n^h-t_m^g|\Big(E_{t_n^h}(x_t^h)+E_{t_n^h}(x_t^g)\Big).
 \end{aligned}
\end{equation}
Integrating \eqref{el4} on the interval $(t_{n-1}^h\vee t_{m-1}^g,t)$ we can estimate
\begin{equation}
\begin{aligned}\label{el5}
 &||x_t^h-x_t^g||_{t_m^g}^2-||x_{t_{n-1}^h\vee t_{m-1}^g}^h-x_{t_{n-1}^h\vee t_{m-1}^g}^g||^2_{t_m^g}\\
  &\leq h\int_{t_{n-1}^h\vee t_{m-1}^g}^{t}\Big(-||\dot x_r^h||^2_{t_n^h}+\frac{E_{t_n^h}(x_{n-1}^h)-E_{t_n^h}(x_{n}^h)}h\Big)\,dr\\
  &+g\int_{t_{n-1}^h\vee t_{m-1}^g}^{t}\Big(-||\dot x_r^g||_{t_m^g}^2+\frac{E_{t_m^g}(x_{m-1}^g)-E_{t_m^g}(x_{m}^g)}g\Big)\,dr\\
  &+L (h\wedge g)\int_{t_{n-1}^h\vee t_{m-1}^g}^t\Big(||\dot x_r^h||_{t_m^g}^2dr+||x_r^h-x_r^g||^2_{t_m^g}\Big)\,dr\\
  &+C_1(h\wedge g)\int_{t_{n-1}^h\vee t_{m-1}^g}^{t}\Big(E_{t_n^h}(x_r^h)+E_{t_n^h}(x_r^g)\Big)\,dr.
 \end{aligned}
\end{equation}
Summing over the partition $\{t_j^{h,g}\}_{j=0}^{N_h+N_g}=\{t_n^h\}_{n=0}^{N_h}\cup\{t_m^g\}_{m=0}^{N_g}$ and exploiting the Lipschitz property of
$t\mapsto ||\cdot||_t$
\begin{align*}
 ||x_t^h-x_t^g||_{t_m^g}^2
 \leq&||x_0^h-x_0^g||_0^2+L(h\wedge g)\sum_{j=1}^{n+m}||x_{t_{j-1}^{h,g}}^h-x_{t_{j-1}^{h,g}}^g||^2_{g(t_{j-1}^{h,g})}\\
  &+\sum_{j=1}^{n}\int_{t_{j-1}^h}^{t_j^h}(-h||\dot x_r^h||^2_{t_j^h}+E_{t_j^h}(x_{j-1}^h)-E_{t_j^h}(x_{j}^h))\,dr\\
  &+\sum_{j=1}^{m}\int_{t_{j-1}^g}^{t_j^g}(-g||\dot x_r^g||^2_{t_j^g}+E_{t_j^g}(x_{j-1}^g)-E_{t_j^g}(x_{j}^g))\,dr\\
  &+L (h\wedge g)\int_0^t\Big(||\dot x_r^h||_{g(r)}^2dr+||x_r^h-x_r^g||^2_{g(r)})\,dr\\
  &+C_1(h\wedge g)\int_{0}^{t}(E_{h(r)}(x_r^h)+E_{h(r)}(x_r^g))\,dr.
\end{align*}
Applying once more the Lipschitz property of $t\mapsto E_t(x)$, we can further estimate
\begin{equation}
\begin{aligned}\label{el6}
 &||x_t^h-x_t^g||_{t_m^g}^2
 \leq||x_0^h-x_0^g||_0^2+L(h\wedge g)\sum_{j=1}^{n+m}||x_{t_{j-1}^{h,g}}^h-x_{t_{j-1}^{h,g}}^g||^2_{g(t_{j-1}^{h,g})}\\
  &+ h(E_{0}(x_{0}^h)-E_{t_n^h}(x_n^h))
  +g(E_{0}(x_{0}^g)-E_{t_m^g}(x_m^g))\\
  &+L(h\wedge g)\int_0^t||\dot x_r^h||_{h(r)}^2dr+C(h\wedge g)\int_{0}^{t}||x_r^h-x_r^g||^2_{g(r)}dr\\
  &+C_1(h\wedge g)\int_{0}^{t}(E_{h(r)}(x_r^h)+E_{h(r)}(x_r^g))dr\\
  &+C_1\sum_j\int_{t_{j-1}^h}^{t_j^h}(t_j^h-t_{j-1}^h)E_{t_j^h}(x_{j-1}^h)dr+C_1\sum_k\int_{t_{k-1}^g}^{t_k^g}(t_k^g-t_{k-1}^g)E_{t_k^g}(x_{k-1}^g)dr.
 \end{aligned}
\end{equation}
From the positivity of $E$ and from \eqref{est2} as well as \eqref{est} we can deduce (with varying constants)
\begin{equation}
\begin{aligned}\label{el7}
 ||x_t^h-x_t^g||_{t_m^g}^2
  \leq&||x_0^h-x_0^g||_0^2+L(h\wedge g)\sum_{j=1}^{n+m}||x_{t_{j-1}^{h,g}}^h-x_{t_{j-1}^{h,g}}^g||^2_{g(t_{j-1}^{h,g})}\\
  &+ hE_{0}(x_{0}^h)
  +gE_{0}(x_{0}^g)\\
  &+C(h\wedge g)+C(h\wedge g)\int_{0}^{t}||x_r^h-x_r^g||^2_{g(r)}dr\\
  &+C(h\wedge g)+Ch+Cg\\
  \leq&||x_0^h-x_0^g||_0^2+L(h\wedge g) \sum_{j=1}^{n+m}||x_{t_{j-1}^{h,g}}^h-x_{t_{j-1}^{h,g}}^g||^2_{g(t_{j-1}^{h,g})}\\
  &+C(h+g)+C(h\wedge g).
 \end{aligned}
\end{equation}
The last inequality follows from
\begin{align*}
 \sup_r|| x_r^{h}||_{t^*}=\sup_n || x_n^{h}||_{t^*}\leq \sup_n (\sqrt{2CnhM}+||x_0^h||_{t^*})\leq \sqrt{2CTM}+||x_0^h||_{t^*},
\end{align*}
where we used the definition of $(x^{h}_r)_{r\in[0,T]}$ in the second equality, triangle inequality and Cauchy-Schwartz inequality in the second, and $C=C(L)$ 
is the constant arising from the log-Lipschitz control \eqref{loglip} of the metric.

For $h,g$ sufficiently small there exists a $\kappa$ satisfying $1-L(h\wedge g)\geq \frac1\kappa>0$.
Applying the discrete Gronwall lemma \cite[Lemma 4.5]{rossi} we finally obtain
\begin{equation}
\begin{aligned}\label{el8}
 &||x_t^h-x_t^g||_{t_m^g}^2
  \leq \kappa C(h+g)e^{\kappa(n+m)(h\wedge g)}\leq \kappa C(h+g)e^{2T\kappa}.
 \end{aligned}
\end{equation}
This shows that if $h_j$ is a vanishing sequence of stepsizes, $\{x^{h_j}\}_j\subset\mathcal C^0([0,T];H)$ is a Cauchy sequence.
 Since $\mathcal C^0([0,T];H)$ is a Banach space there exists a continuous curve $(x_t)_{t\in[0,T]}$ and a subsequence (not relabeled)
 such that $x^{h_j}_t\to x_t$ in $\mathcal C^0([0,T];H)$ as $j\to \infty$. From \eqref{est3} it follows immediately that also $\tilde x^{h_j}_t,\bar x^{h_j}_t$ 
 converge to $x_t$.
 
 Since $\int_0^T||\dot x_r^{h_j}||^2dr\leq M$ we can extract a further subsequence (not relabeled) with
 \begin{equation*}
  \dot x^{h_j}\rightharpoonup u \text{ in }L^2([0,T];H)
 \end{equation*}
where $u$ is some function in $L^2([0,T];H)$. As a consequence we obtain that the limit function $x\in \AC^2([0,T];H)$ since for all 
$0<s<t<T$
\begin{align}\label{modxtxs}
||x_t-x_s||_{t^*}=\lim_{j\to\infty}||x_t^{h_j}-x_s^{h_j}||_{t^*}=\lim_{j\to\infty}||\int_s^t\dot x_r^{h_j}dr||_{t^*}\leq\int_s^t ||u_r||_{t^*}dr,
\end{align}
where $t^*$ is an arbitrarily fixed timepoint in $[0,T]$.
We still have to show that $u_r=\dot x_r$ almost everywhere. This follows again straightforward from the weak convergence of $\dot x^{h_j}$. Let $y\in H$, then
\begin{align*}
 \langle x_t-x_s,y\rangle_{t^*}=\lim\langle x_t^{h_j}-x_s^{h_j},y\rangle_{t^*}=\lim\langle\int_s^t\dot x_r^{h_j}dr,y\rangle_{t^*}=
 \langle\int_s^t u_rdr,y\rangle_{t^*}.
\end{align*}
Since $y\in H$ is arbitrary we obtain
\begin{align*}
x_t-x_s=\int_s^t u_rdr,
\end{align*}
and hence $\lim_{s\to t}\frac{x_t-x_s}{t-s}=u_t$ at every Lebesgue point of $u$.
 
 \end{proof}

\begin{thm}\label{exgradflow}
Let $E$ be as in the beginning of this section. Then
 for every $x\in \Dom(E)$ there exists a unique $(x_t)\in AC^2([0,T];H)$ with $\lim_{t\searrow0}x_t=x$ such that 
 \begin{align*}
  \partial_t x_t\in -D_t^-E_t(x_t)\text{ for a.e. }t\in(0,T).
 \end{align*}
\end{thm}
\begin{proof}
 Recall that the minimizers of \eqref{minprob1} with $x_0^h:=x$ satisfy the Euler-Lagrange equation, that is in terms of the subdifferential of $E$, the
 piecewise linear interpolant $x_t^h$ and the piecewise constant interpolant $\bar x_t^h$ 
 \begin{equation*}
  \langle \dot x_t^h,\bar x_t^h-y\rangle_{t_n^h} +E_{t_n^h}(\bar x_t^h)-E_{t_n^h}(y)\leq 0 \quad\forall y\in H, \text{ for every } t\in(t_{n-1}^h,t_n^h).
 \end{equation*}
The log Lipschitz property together with the polarization identity gives then for all $y\in H$ and almost every $t\in[0,T]$
\begin{equation*}
  \langle \dot x_t^h,\bar x_t^h-y\rangle_{t} +E_{t}(\bar x_t^h)-E_{t}(y)\leq Lh(||\dot x_t^h||_t^2+||\bar x_t^h-y||^2_t)
  +C_1h(e(\bar x_t^h)+e(y)).
 \end{equation*}
 Integrating this inequality over the interval $(s,t)$ for some $0<s<t<T$ we deduce
 \begin{equation}\label{e-l}
 \begin{aligned}
  &\int_s^t\langle \dot x_r^h,\bar x_r^h-y\rangle_{r}\,dr +\int_s^tE_{r}(\bar x_r^h)-E_{r}(y)\,dr\\
  &\leq Lh \int_s^t||\dot x_r^h||^2+||\bar x_r^h-y||_r^2\,dr
  +C_1h\int_s^t(e(\bar x_r^h)+e(y))dr.
 \end{aligned}
 \end{equation}
 Applying Proposition \ref{converge} we get existence of a subsequence and a curve $(x_t)\in \AC^2([0,T];H)$ such that $\bar x_t^h\to x_t$ in 
 $\mathcal C^0([0,T];H)$ and $\dot x_t^h\to \dot x_t$ weakly in $L^2([0,T];H)$. Hence we get for all $y\in \Dom(E)$
 
 \begin{align*}
  &\int_s^t\langle \dot x_r, x_r-y\rangle_{r}dr +\int_s^tE_{r}( x_r)-E_{r}(y)dr\\
  &\leq \liminf_{h\to0} \int_s^t\langle \dot x_r^h,\bar x_r^h-y\rangle_{r}dr +\liminf_{h\to0}\int_s^tE_{r}(\bar x_r^h)-E_{r}(y)dr\\
  &\leq \liminf_{h\to0} \Big\{\int_s^t\langle \dot x_r^h,\bar x_r^h-y\rangle_{r}dr +\int_s^tE_{r}(\bar x_r^h)-E_{r}(y)dr\Big\}\\
  &\leq\liminf_{h\to0}\Big\{Lh \int_s^t||\dot x_r^h||^2+||\bar x_r^h-y||_r^2\,dr
  +C_1h\int_s^t(e(\bar x_r^h)+e(y))dr\Big\}\\
  &\leq0,
 \end{align*}
 where we applied Fatou's Lemma and the lower semicontinuity of $x\mapsto E_t(x)$ in the first inequality, estimate \eqref{e-l} in the third inequality and 
 the non-negativity of $E_t(y)$,
 \eqref{est} and \eqref{est2} in the last. Dividing by $t-s$ and letting $s\to t$ we infer from the Lebesgue differentiation theorem
  that 
  \begin{align*}
  \langle \dot x_t, x_t-y\rangle_{t} +E_{t}( x_t)-E_{t}(y)\leq 0
  \end{align*}
for almost every $t\in(0,T)$ and $y\in X$.

Since $\bar x_t^h$ converges to $x_t$ for every $t$ we clearly have that $\lim_{t\searrow0}x_t=x$.

Suppose there exists two absolutely continuous curves $(x_t)$, $(\tilde x_t)_{t\in[0,T]}$ such that for every $y\in X$ and almost every $t\in[0,T]$
\begin{align*}
\langle \dot x_t, x_t-y\rangle_{t} +E_{t}( x_t)-E_{t}(y)&\leq 0,\\
 \langle \dot {\tilde x}_t, \tilde x_t-y\rangle_{t} +E_{t}(\tilde x_t)-E_{t}(y)&\leq 0
\end{align*}
with $\lim_{t\searrow0}x_t=\lim_{t\searrow 0}\tilde x_t=x_0$.
Inserting $\tilde x_t$ for $y$ into the first inequality and $x_t$ for $y$ into the second we obtain by adding and using \eqref{monotone}
\begin{align*}
\partial_s\frac12{||x_s-\tilde x_s||^2_{t}}\Big|_{s=t}=\langle \dot x_t-\dot{\tilde x}_t, x_t-\tilde x_t\rangle_{t} \leq 0.
\end{align*}
From the log-Lipschitz continuity of the metric we deduce
\begin{align*}
\partial_s\frac12{||x_s-\tilde x_s||^2_{s}}\Big|_{s=t} \leq L||x_t-\tilde x_t||_t^2.
\end{align*}
Applying Gronwall's inequality we conclude
$||x_t-\tilde x_t||_t^2\leq e^{2Lt}||x_0-\tilde x_0||^2_0=0$ for almost every $t\in[0,T]$ and hence for every $t\in[0,T]$ by continuity. This proves uniqueness.
\end{proof}

\section{The heat equation on time-dependent metric measure spaces}\label{secheat}

Let $(X,d_t,m_t)_{t\in[0,T]}$ be a family of Polish metric measure space. We always assume that \eqref{loglip} holds and 
that there exists a reference measure $m\in\mathcal{P}(X)$ such that $m_t=e^{-f_t}m$ with Borel
functions $f_t$ satisfying $|f_t(x)|\leq C$ and
\begin{equation}\label{lipschitzinspace}
 |f_t(x)-f_s(x)|\leq L^*|t-s|,\quad |f_t(x)-f_t(y)|\leq Cd_t(x,y).
\end{equation}

Let us denote \textit{Cheeger's energy} by $\Ch_t\colon L^2(X,m_t)\to[0,\infty]$ 
\begin{equation*}
 \Ch_t(u)=\frac12\inf\left\{\liminf_{n\to\infty}\int_X(\lip_tu_n)^2dm_t|u_n\in \Lip(X),\int_X|u_n-u|^2dm_t\to0\right\},
\end{equation*}
where $\lip_tu$ denotes the local slope defined by
\begin{align*}
 \lip_tu(x):=\limsup_{y\to x}\frac{|u(x)-u(y)|}{d_t(x,y)}.
\end{align*}
By making use of the \textit{minimal relaxed gradient} $|\nabla_tu|_*$ (\cite[Definition 4.2]{agscalc}),
this functional admits the integral representation
\begin{equation*}
 \Ch_t(u)=\frac12\int_X|\nabla_tu|_*^2dm_t,
\end{equation*}
set equal to $+\infty$ if $u$ has no relaxed slope. This defines a convex and lower semicontinuous functional in $L^2(X,m_t)$ \cite[Theorem 4.5]{agscalc}.

\begin{lma}\label{relgradient}
Fix $t\in[0,T]$ and let $u\in \Dom(\Ch_t)$. Then, asssuming \eqref{loglip}
\begin{equation*}
  |\nabla_tu|_*\leq e^{L|t-s|}|\nabla_su|_* \quad m\text{-a.e. in }X, \quad\forall s\in[0,T].
\end{equation*}

\end{lma}
\begin{proof}
 Since $u\in \Dom(\Ch_t)$ we know $u\in Dom(\Ch_s)$ as well and there exist bounded Borel Lipschitz functions $u_n\in L^2(X,m_s)$ such that 
 \begin{equation*}
  u_n\to u, \quad\lip_s u_n\to |\nabla_s u|_* \text{ strongly in }L^2(X,m_s),
 \end{equation*}
 see e.g. \cite[Lemma 4.3 (c)]{agscalc}. This implies that $e^{L|t-s|}|\nabla_su|_*$ is a relaxed $d_t$-gradient since 
 \begin{equation*}
  u_n\to u, \quad e^{L|t-s|}\lip_s u_n\to e^{L|t-s|}|\nabla_s u|_* \text{ strongly in }L^2(X,m_t)
 \end{equation*}
and
\begin{equation*}
 |\nabla_tu_n|_*\leq e^{L|t-s|}\lip_su_n,
\end{equation*}
c.f. \cite[Lemma 4.3. (a)]{agscalc}. Thus Lemma 4.4 in \cite{agscalc} yields the assertion.

\end{proof}

Due to our assumptions the sets $L^2(X,m_t)$ and $ \Dom(\Ch_t)$ do not depend on $t$. 
The domain $\Dom(\Ch)$  of Cheeger's energy endowed with the norm
$$\sqrt{||u||^2_{L^2(X,m_t)}+\Ch_t(u)}$$
is a Banach space, cf. \cite[Theorem 2.7]{cheeger}. 

In the following we additionally impose that for each $t$ the space $(X,d_t,m_t)$ is \emph{infinitesimally Hilbertian}, i.e. Cheeger's energy $\Ch_t$ 
defines a 
quadratic form. In particular the domain is a separable Hilbert space and Lipschitz functions are dense, see \cite{agsmet}. 
In this case we will denote by $\mathcal E_t$ the associated \textit{Dirichlet form}, which is the unique 
bilinear symmetric form satisfying
\begin{equation*}
 \mathcal{E}_t(u,u)=2\Ch_t(u) \quad\forall u\in \Dom(\Ch).
\end{equation*}
Moreover $\mathcal E_t$ is strongly local \cite[Proposition 4.14]{agsmet}, i.e.
\begin{align*}
 u,v\in\Dom(\Ch), v\text{ constant on }\{u\neq0\}\Rightarrow \mathcal E(u,v)=0,
\end{align*}
and admits the integral representation
\begin{equation*}
 \mathcal{E}_t(u,v)=\int\nabla_t u\cdot\nabla_t v\, dm_t      \quad u,v\in \Dom(\Ch),
\end{equation*}
where
\begin{equation*}
 \nabla_t u\cdot\nabla_t v:=\lim_{\varepsilon\searrow0}\frac{|\nabla_t(u+\varepsilon v)|^2_*-|\nabla_tu|^2_*}{2\varepsilon}
\end{equation*}
and the limit is understood in $L^1(X,m_t)$, see \cite[Proposition 4.14]{agsmet}.

We define the Laplace operator $\Delta_t$ as the generator of $\mathcal E_t$, i.e. as the unique non-positive self adjoint operator on $L^2(X,m_t)$ with domain
$\Dom(\Delta_t)\subset \Dom(\Ch)$ and 
\begin{equation*}
 -\int_X\Delta_tu v\, dm_t=\mathcal{E}_t(u,v) \quad \forall u\in \Dom(\Delta_t), v\in \Dom(\Ch).
\end{equation*}

We set $\mathcal{F}=\Dom(\Ch)$ endowed with $||u||_\F^2:={||u||^2_{L^2(X,m_t)}+\E_t(u)}$ and $\mathcal{H}=L^2(X,m_t)$. We identify $\mathcal H$ with its own dual; the dual of $\F$ is denoted by $\F^*$. In particular we have $\F\subset \mathcal H\subset \F^*$ with dense and continuous embeddings.
We define for $0\leq s<\tau\leq T$ the Hilbert space
\begin{equation*}
 \mathcal{F}_{(s,\tau)}=L^2((s,\tau)\to\mathcal F)\cap H^1((s,\tau)\to\mathcal F^*),
\end{equation*}
equipped with the norm $(\int_s^\tau||u_t||^2_\F+||\partial_tu_t||_{\F^*}^2\,dt)^{1/2}$. According to Lemma 10.3 in \cite{renardy2006introduction} we have
$\F_{(s,\tau)}\subset\mathcal C([s,\tau]\to \mathcal H)$.
\begin{Def}\label{defheat}
 A function $u$ is called \emph{solution to the heat equation}
 \begin{equation*}
  \partial_t u=\Delta_t u \text{ on }(s,\tau)\times X
 \end{equation*}
 if $u\in\mathcal{F}_{(s,\tau)}$ and if for all $w\in\mathcal{F}_{(s,\tau)}$
 \begin{equation}
  -\int_s^\tau\mathcal{E}_t(u_t,w_t)dt=\int_s^\tau\langle\partial_t u_t, w_te^{-f_t}\rangle_{\mathcal F^*,\mathcal F}dt,
 \end{equation}
 where $\langle\cdot,\cdot\rangle_{\mathcal F^*,\mathcal F}$ denotes the dual pairing.
 
A function $v$ is called \emph{solution to the adjoint heat equation}
$$-\Delta_s v +\partial_s f \cdot v=\partial_s v\qquad\mbox{on }(\sigma,t)\times X$$
if $v\in \F_{(\sigma,t)}$ and if for all $w\in \F_{(\sigma,t)}$
$$\int_\sigma^t \E_s(v_s,w_s)ds+\int_\sigma^t \int_X v_s\cdot w_s\cdot \partial_sf_s\,dm_s\,ds=\int_\sigma^t \langle \partial_sv_s, w_se^{-f_s}\rangle_{\F,\F^*}\,ds.$$
 \end{Def}

By virtue of Theorem 2.2 and Theorem 2.5 in \cite{sturm2016} we have
existence and uniqueness to solutions of the heat and the adjoint heat equation with initial condition 
$u_s=h\in \mathcal H$ and terminal condition $v_t=h\in\mathcal H$ respectively.
We denote these solutions by
 \begin{equation*}
  u_t(x)=P_{t,s}h(x), \qquad v_s(x)=P^*_{t,s}h(x).
 \end{equation*}
Both solutions, called heat flow and adjoint heat flow respectively, satisfy
\begin{align*}
 P_{t,s}h(x)&=P_{t,r}\circ P_{r,s} h(x),\\
  P_{t,s}^*h(y)&=P^*_{r,s}\circ P^*_{t,r} h(x).
\end{align*}
The operators are dual to each other in the sense that
\begin{align*}
\int( P_{t,s}u )v\, dm_t=\int uP^*_{t,s}v\, dm_s,
\end{align*}
and 
\begin{equation}
\begin{aligned}\label{ddual}
&\text{if }h\in \mathcal H \text{ with } 0\leq h\leq 1 \quad\text{ then }\quad0\leq P_{t,s}h\leq 1\\
&P_{t,s}1=1\quad \text{ whenever } \quad m(X)<\infty,
\end{aligned}
\end{equation}
cf. Section 2 in \cite{sturm2016}.
\subsection{Identification of the forward adjoint heat flow with the dynamic EDI-gradient flow for the entropy}
We consider the adjoint heat flow $(\rho_t)_{0\leq t\leq T}$ parametrized forwards in time, i.e. solving
\begin{equation*}
 \partial_t\rho_t=\Delta_t\rho_t+\rho_t\partial_tf_t\quad \text{ on }(0,T)\times X
\end{equation*}
with nonnegative initial data $\rho_0=h$. In the following we show coincidence of $(\rho_t)$ with the dynamic EDI-gradient flow $(\mu_t)$ of $S$ via $\mu_t=\rho_tm_t$. For this we have to assume that each $(X,d_t,m_t)$ satisfies CD$(K,\infty)$.
Following the approach in \cite{agscalc}, we prove that
$\mu_t=\rho_tm_t$ is a dynamic EDI-gradient flow of $S$. From the uniqueness it follows that both flows coincide.
\begin{lma}
Let $h\in\mathcal H$ and $(\rho_t)$ be the solution to the forward adjoint heat flow on $(0,T)\times X$ with $\rho_0=h$.
\begin{enumerate}
 \item  The flow $(\rho_t)$ is mass preserving, i.e.
 \begin{equation}\label{masscons}
  \int\rho_t\,dm_t=\int h\,dm_0 \quad \forall 0\leq t\leq T.
   \end{equation}
 \item If $e\colon\mathbb{R}\to[0,\infty]$ is a convex lower semicontinuous function and $e'$ is locally Lipschitz in $\mathbb R$,
 it holds for a.e. $t\in(0,T)$
 \begin{equation}\label{dissipation}
  \frac{d}{dt}\int e(\rho_t)\,dm_t=-\int e''(\rho_t)|\nabla_t\rho_t|^2_*\,dm_t+\int\partial_t f_t(\rho_te'(\rho_t)-e(\rho_t))\,dm_t.
 \end{equation}

\end{enumerate}

\end{lma}
\begin{proof}
 Since the measure is finite, $1\in\mathcal H$, and hence by duality
 \begin{align*}
 \int\rho_t\, dm_t=\int h\, dm_0.
 \end{align*}

In order to prove \eqref{dissipation} we assume by a standard approximation that $e'$ is bounded and globally Lipschitz, cf. \cite[Theorem 4.16]{agscalc}.
Since $e$ is convex and $\rho\in\F_{0,T}$ we have for $t_0<t_1$
\begin{align*}
 &\int e(\rho_{t_1})\,d m_{t_1}-\int e(\rho_{t_0})\,dm_{t_0}\\
 \geq &\int e'(\rho_{t_0})(\rho_{t_1}-\rho_{t_0})\,dm_{t_1}
 +\int e(\rho_{t_0})\,d(m_{t_1}-m_{t_0})\\
 =&\int_{t_0}^{t_1}\langle \partial_t\rho_t,e'(\rho_{t_0})e^{-f_{t_1}}\rangle_{\F^*,\F}\,dt-\int_{t_0}^{t_1}\int e(\rho_{t_0})\partial_t f_t\,dm_t\,dt\\
 \geq &\int_{t_0}^{t_1}\Big(
 -\frac12||\partial_t\rho_t||_{\F^*}^2-\frac12||e'(\rho_{t_0})e^{-f_{t_1}}||^2_\F-\int e(\rho_{t_0})\partial_t f_t\,dm_t\Big)\,dt,
\end{align*}
which is integrable. Changing the roles of $t_0$ and $t_1$ shows that $s\mapsto\int e(\rho_t)\,dm_t$ is absolutely continuous. 
Then, since $\rho\in \F_{(0,T)}$, we deduce from the mean value theorem for a.e. $t$
\begin{align*}
 &\lim_{h\to0}\frac1h\Big(\int e(\rho_{t+h})\,dm_{t+h}-\int e(\rho_t)\,dm_t\Big)\\
 =&\lim_{h\to0}\frac1h\int (e(\rho_{t+h})-e(\rho_t))e^{-f_{t+h}}\,dm+\lim_{h\to0}\frac1h\int e(\rho_t)(e^{-f_{t+h}}-e^{-f_t})\,dm\\
 =&\lim_{h\to0}\int e'(\rho_t)\frac{\rho_{t+h}-\rho_t}h \,dm_t-\int e(\rho_t)\partial_tf_t\,dm_t\\
 =&\langle\partial_t\rho_t,e'(\rho_t)e^{-f_t}\rangle_{\F^*,\F}-\int e(\rho_t)\partial_tf_t\,dm_t,
 \end{align*}
 cf. \cite[Corollary 5.5]{lierl2015}, \cite[Lemma 12.3]{ams}.
 Since $\rho$ is a solution to the forward adjoint heat equation we have 
 \begin{align*}
 \langle\partial_s\rho_s,e'(\rho_s)e^{-f_s}\rangle_{\F^*,\F}
=& -\E_s(\rho_s,e'(\rho_s))+\int \rho_se'(\rho_s)\partial_s f_s\,dm_s\\
=&-\int e''(\rho_s)|\nabla_s\rho_s|_*^2dm_s+\int \rho_se'(\rho_s)\partial_s f_s\,dm_s,
\end{align*}
which proves \eqref{dissipation}.
\end{proof}

\begin{Prop}\label{deviatentropy}
Let $(\rho_t)_{0\leq t\leq T}$ be the solution of the forward adjoint heat equation with nonnegative initial datum
$h\in \mathcal H$.
Then it holds
\begin{equation}
\begin{aligned}\label{fisherest}
  \int_0^t\int_{\{\rho_r>0\}}\frac{|\nabla_r\rho_r|_*^2}{\rho_r}\,dm_r\,dr
  \leq \int h\log h\,dm_0+\int h\,dm_0\\
  -m_t(X)+\int_0^t\int(\partial_rf_r)\rho_r \,dm_r\,dr,
 \end{aligned}
 \end{equation}
 and the map $t\mapsto \int\rho_t\log\rho_t dm_t$ is locally absolutely continuous in $[0,T]$ and 
\begin{equation}\label{eq:deviatentropy}
 \frac{d}{d t}\int\rho_t\log\rho_t \,dm_t=-\int_{\{\rho_t>0\}}\frac{|\nabla_t\rho_t|_*^2}{\rho_t}\,dm_t+\int(\partial_t f_t) \rho_t \,dm_t
\end{equation}
for a.e. $t\in[0,T]$.
\end{Prop}
\begin{proof}
By duality and \eqref{ddual}, we have $\rho_t\geq0$ for every $t\in(0,T)$.
Applying formula \eqref{dissipation} to $\rho_t+\varepsilon$ we get
 \begin{equation*}
  \frac{d}{dt}\int(\rho_t+\varepsilon)\log(\rho_t+\varepsilon)\,dm_t=-\int\frac{|\nabla_t\rho_t|^2_*}{\rho_t+\varepsilon}\,dm_t+\int\partial_t f_t(\rho_t+\varepsilon)\,dm_t.
 \end{equation*}
Integrating from $0$ to $t$ and letting $\varepsilon$ go to 0, we obtain by applying dominated and monotone convergence
\begin{equation}
\begin{aligned}\label{eq:deviatentropy1}
 &\int\rho_t\log\rho_t\,dm_t-\int\rho_0\log\rho_0\,dm_0\\
 =&\int_0^t-\int_{\{\rho_r>0\}}\frac{|\nabla_r\rho_r|^2_*}{\rho_r}\,dm_r+\int(\partial_r f_r)\rho_r\,dm_r\,dr.
\end{aligned}
\end{equation}
Using $\rho\log\rho\geq \rho-1$ and the conservation of total mass \eqref{masscons} leads to 
 \begin{align*}
  \int_0^t\int_{\{\rho_r>0\}}\frac{|\nabla_r\rho_r|^2_*}{\rho_r}\,dm_r\,dr\leq \int h\log h\,dm_0+\int h\,dm_0\\
  -m_t(X)+\int_0^t\int(\partial_r f_r)\rho_r\,dm_r\,dr,
 \end{align*}
which proves \eqref{fisherest}. 
As a consequence from \eqref{fisherest} and \eqref{eq:deviatentropy1} we get the local absolute continuity of $s\mapsto \int\rho_s\log\rho_s dm_s$ and 
\eqref{eq:deviatentropy}.
\end{proof}

The following two lemmas are crucial to conclude that the forward adjoint heat flow defines the EDE-gradient flow for the relative entropy. The first lemma 
gives an estimate of the squared slope of the entropy in terms of the Fisher information, which is an estimate in the static setting, while the second lemma
represents a dynamic version of Kuwada's Lemma, see e.g. \cite[Proposition 3.7]{gko}. The proof of Proposition \ref{kuwadaslemma} relies on the dual formula 
of the dynamic distance $W_{s,t}$ (recall Definition \ref{ddist}) in terms of subsolutions to a modified Hamilton-Jacobi equation, cf. \cite[Section 6]{sturm2016}.

\begin{Prop}\label{squaredslope}
Assume $(X,d_t,m_t)$ satisfies CD$(K,\infty)$.
For $\mu=\rho m_t\in \Dom(S)$ 
\begin{equation*}
|\nabla_tS_t|^2(\mu)\leq\int_{\{\rho>0\}}\frac{|\nabla_t\rho|_*^2}{\rho}dm_t.
\end{equation*}
\end{Prop}
\begin{proof}
This is due to Theorem 9.3 in \cite{agscalc}.
\end{proof}

\begin{Prop}\label{kuwadaslemma}
 Let $(\rho_t)_{0\leq t\leq T}$ be the solution to the forward adjoint heat equation with nonnegative initial datum $h\in \mathcal H$ such that
  $\int h\,d m_0 =1$.
Then the curve $t\mapsto\mu_t:=\rho_t m_t$ is locally absolutely continuous and satisfies
\begin{equation*}
 |\dot\mu_t|_t^2\leq\int_{\{\rho_t>0\}}\frac{|\nabla_t\rho_t|_*^2}{\rho_t}\,dm_t \quad\text{ for a.e. }t\in[0,T].
\end{equation*}

\end{Prop}
\begin{proof}
From \eqref{masscons} we know that $\int\rho_t\,dm_t=1$ for every $0\leq t\leq T$. Hence each $\mu_t=\rho_tm_t$ is a probability measure.

Let $s<t$ and set $\delta:=t-s$. Then, with $\vartheta(a)=s+a$, we define 
HLS$_\vartheta$ as in \cite[Section 6]{sturm2016} by
\begin{align*}
 \text{HLS}_{\vartheta}:=\bigg\{\varphi\in{\Lip}_b([0,\delta]\times X)\bigg|
 \ &\partial_a\varphi_a\leq -\frac12|\nabla_{\vartheta(a)}(\varphi_a)|_*^2\\
 &\mathcal L\times m\text{ a.e. in }(0,\delta)\times X\bigg\},
\end{align*}
and
\begin{equation*}
\tilde W^2_{\vartheta}(\mu_s,\mu_t):=2\sup_\varphi\left\{\int\varphi_{\delta}d\mu_t-\int\varphi_{0}d\mu_s\right\},
\end{equation*}
where the supremum runs over all maps $\varphi(a,x)=\varphi_a(x)\in$ HLS$_{\vartheta}$. Then we
have by Lemma 6.5 in \cite{sturm2016}
\begin{align*}
 W_s^2(\mu_s,\mu_t)\leq e^{2L\delta}\delta\tilde W^2_\vartheta(\mu_s,\mu_{t}).
\end{align*}
By applying \cite[Lemma 4.3.4]{ags} to the function $(a,b)\mapsto \int\rho_a\varphi_bdm_a$, where $\varphi\in$ HLS$_\vartheta$, we obtain
\begin{align*}
 &\int\varphi_\delta\,d\mu_t-\int\varphi_0\,d\mu_s=\int_0^\delta\partial_a\int\varphi_a\,d\mu_{s+a}\,da\\
 &\leq\int_0^\delta\int-\frac12|\nabla_{s+a}(\varphi_a)|_*^2\,d\mu_{s+a}-\mathcal E_{s+a}(\rho_{s+a},\varphi_a)\,da\\
 &\leq \int_0^\delta\int-\frac12|\nabla_{s+a}(\varphi_a)|_*^2\,d\mu_{s+a}\\
 &+\int\frac12|\nabla_{s+a}(\varphi_a)|_*^2\,d\mu_{s+a}+\frac12\int_{\{\rho_{s+a}>0\}}\frac{|\nabla_{s+a}(\rho_{s+a})|_*^2}{\rho_{s+a}}\,dm_{s+a}\,da\\
 &=\int_0^\delta\frac12\int_{\{\rho_{s+a}>0\}}\frac{|\nabla_{s+a}(\rho_{s+a})|_*^2}{\rho_{s+a}}\,dm_{s+a}\,da.
\end{align*}
Taking the supremum over all $\varphi$
\begin{align*}
 W_s^2(\mu_s,\mu_t)\leq e^{2L\delta}\delta\int_0^\delta\int_{\{\rho_{s+a}>0\}}\frac{|\nabla_{s+a}(\rho_{s+a})|_*^2}{\rho_{s+a}}\,dm_{s+a}\,da.
\end{align*}
Dividing by $\delta^2$ and letting $\delta\to0$ we conclude
\begin{align*}
 |\dot\mu_s|_s^2\leq \int_{\{\rho_{s}>0\}}\frac{|\nabla_s\rho_s|_*^2}{\rho_s}\,dm_s.
\end{align*}

\end{proof}
Now we are ready to prove our main result of this section.

\begin{thm}\label{identico}
Let $(X,d_t,m_t)_{t\in[0,T]}$ be a family of Polish spaces with complete geodesic distances $d_t$ satisfying \eqref{loglip} such that
$m_t=e^{-f_t}m$, where $m\in\mathcal{P}(X)$ and $f_t$ are bounded functions satisfying and \eqref{lipschitzinspace}.
Assume that each static space satisfies RCD$(K,\infty)$ for some finite number $K\in\mathbb R$.
Let $h\in \mathcal{H}$ nonnegative with $\bar\mu=hm_0$. 
 \begin{enumerate}
  \item Let $\rho_t$ solve the forward adjoint heat equation starting from $h$, then $\mu_t=\rho_tm_t$ is the dynamic EDE-gradient flow for the relative 
  entropy $S_t$ starting in $\bar\mu$.
  \item Conversely, let $\mu_t$ be the dynamic EDE-gradient flow for $S_t$, then $\mu_t=\rho_tm_t$ and $\rho_t$ is the solution to the  forward adjoint heat 
  equation.
 \end{enumerate}
 \end{thm}
\begin{proof}
Proposition \ref{deviatentropy} applied the forward flow $\rho_t$ yields
 \begin{equation*}
 \frac{d}{dt}\int\rho_t\log\rho_t \,dm_t=-\int_{\{\rho_t>0\}}\frac{|\nabla_t\rho_t|_*^2}{\rho_t}\,dm_t+\int(\partial_t f_t)\rho_t \,dm_t.
\end{equation*}

Integrating from 0 to $t$ and using Proposition \ref{kuwadaslemma} and Proposition \ref{squaredslope} we obtain
\begin{equation*}
  S_t(\mu_t)+\frac12\int_0^t|\dot{\mu}_r|^2_rdr+\frac12\int_0^t|\nabla_rS_r|^2(\mu_r)dr
  \leq S_0(\bar\mu)+\int_0^t(\partial_r S_r)(\mu_r)dr.
 \end{equation*}

Moreover, by virtue of Proposition 2.8 in \cite{sturm2016}, $(\mu_t)$ is 
contained in the sublevel set of the entropy and hence, similarly as in the proof of Theorem \ref{uniqueness}, we get for all $t$
\begin{align*}
  S_t(\mu_t)-S_0(\bar\mu)\geq\int_0^t(\partial_rS_r)(\mu_r)dr-\int_0^t|\dot{ \mu}|_r|\nabla_rS_r|(\mu_r)dr.
\end{align*}
Thus we have
\begin{equation*}
  S_t(\mu_t)+\frac12\int_0^t|\dot{\mu}_r|^2_rdr+\frac12\int_0^t|\nabla_rS_r|^2(\mu_r)dr
  = S_0(\bar\mu)+\int_0^t(\partial_r S_r)(\mu_r)dr.
 \end{equation*}
To show the converse implication, let 
$\tilde\rho_t$ be the solution to the adjoint heat equation parametrized forwards in time. 
From the previous argumentation we know that $\tilde \mu_t=\tilde\rho_tm_{t}$ is a dynamic
EDE-gradient flow of the entropy. 
From Theorem \ref{uniqueness} there is at most one gradient flow 
starting from $\bar\mu$, hence $\tilde\mu_t=\mu_t$ for every $t\in[0,T]$. 
\end{proof}
\begin{remark}\label{vr}
Let us recall the complete picture of forward and backward equation described in Section \ref{secheat}. The heat equation (forward in time) induces the adjoint heat equation (backward in time)
and vice versa. Then $\mu_s:=\rho_sm_s$, where $\rho_s$ denotes the adjoint heat flow (backward in time) is an \emph{upward dynamic EDI-gradient flow} in the sense that
\begin{align*}
 S_s(\mu_s)+\frac12\int_s^T|\dot\mu_r|_r^2\, dr+\frac12\int_s^T|\nabla_rS_r(\mu_r)|^2\, dr= S_T(\mu_T)+\int_s^T(\partial_rS_r)(\mu_r)\, dr.
\end{align*}
Equivalently, and this is what we showed, if $\mu_t=\rho_tm_t$, where $\rho_t$ solves the adjoint heat equation forward in time, then $\mu_t$ solves
\begin{align*}
 S_t(\mu_t)+\frac12\int_0^t|\dot\mu_r|_r^2\, dr+\frac12\int_0^t|\nabla_rS_r(\mu_r)|^2\, dr= S_0(\mu_0)+\int_0^t(\partial_rS_r)(\mu_r)\, dr.
\end{align*}
But then the heat equation is a backward equation.
\end{remark}
\begin{remark}
If each $(X,d_t,m_t)$ is supposed to be RCD$(K,N)$ space for finite numbers $K,N$ and $(X,d_t,m_t)_{t\in[0,T]}$ is a super-Ricci flow (see \cite{sturm2015,sturm2016}) then it is shown in \cite{sturm2016} that $\mu_s=\rho_sm_s$ is characterized as the unique backward EVI$(-2L,\infty)$-gradient flow of the relativ entropy, i.e.
\begin{align*}
 \frac12\partial_s^-W_{s,t}(\mu_s,t)^2\Big|_{s=t}+LW_t^2(\mu_t,\sigma)\geq S_t(\mu_t)-S_t(\sigma).
\end{align*}
Then Proposition \ref{propevi} already implies that $\mu_t=\rho_tm_t$ is a dynamic EDE-gradient flow.
\end{remark}

\subsection{Identification of the heat flow with the dynamic gradient flow for Cheeger's energy}

In the following let $(X,d_t,m_t)_{t\in[0,T]}$ be a family of Polish metric measure spaces. We suppose that $(d_t)$ satisfies \eqref{loglip} and $m_t=e^{-f_t}m$,
where $m$ is a $\sigma$-finite Borel measure on $X$ and $(f_t)$ are Borel functions satisfying
\begin{equation}\label{flipch}
 |f_t(x)-f_s(y)|\leq L^*|t-s|.
\end{equation}
We consider Cheeger's energy $\Ch_t\colon L^2(X,m_t)\to[0,\infty]$, defined by
\begin{equation*}
 \Ch_t(u)=\frac12\int_X |\nabla_t u|_*^2dm_t,
\end{equation*}
where $|\nabla_t u|_*$ denotes the minimal relaxed gradient of $u$. 
Since $L^2(X,m_t)$ is a separable Hilbert space and the assumptions on the energy functional from Section \ref{sechilbert} are satisfied by $(\Ch_t)_t$ we directly obtain existence of 
a gradient flow in the sense of Definition \ref{defgradflowhilbert}.

\begin{thm}\label{exi}
 Let $\bar u\in \Dom(\Ch)$. Then there exists a unique gradient flow for $\Ch$ starting in $\bar u$, i.e. an absolutely continuous curve $(u_t)_t$ solving
 \begin{align*}
  \partial_t u_t\in-D^-_t\Ch_t(u_t)\quad \text{ for a.e. }t\in(0,T)
 \end{align*}
 and $\lim_{t\to0}u_t=\bar u$.

\end{thm}
\begin{proof}
Obviously $\Ch_t\geq0$ for every $t\in[0,T]$. Moreover $u\mapsto\Ch_t(u)$ is convex and lower semicontinuous by Theorem
4.5 in \cite{agscalc}. From Lemma \ref{relgradient} and \eqref{flipch} we obtain
\begin{align*}
 &|\Ch_t(u)-\Ch_s(u)|\leq |\int|\nabla_tu|^2_*-|\nabla_su|^2_*\,dm_t|+|\int|\nabla_su|^2_*\,d(m_t-m_s)|\\
 &\leq 2L|t-s|\int|\nabla_su|^2_*\,dm_t+L^*e^{L^*|t-s|}|t-s|\int|\nabla_su|^2_*\,dm_s\\
 &\leq 2L|t-s|e^{C|t-s|}\int|\nabla_su|^2_*\,dm_s+L^*e^{L^*|t-s|}|t-s|\int|\nabla_su|^2_*\,dm_s\\
 &\leq (2L+L^*)e^{L^*|t-s|}|t-s|\Ch_s(u).
\end{align*}
We get the result as a consequence of Theorem \ref{exgradflow}.
\end{proof}

In the case when the underlying spaces are infinitesimally Hilbertian we have by virtue of Theorem 2.2 in \cite{sturm2016} a unique solution to the heat equation $\partial_tu_t=\Delta_tu_t$ on $(0,T)\times X$ for each initial condition. Note that the solution is a priori `only' contained in $\F_{(0,T)}$ and hence $\Delta_tu_t$ is not an element in $L^2(X,m)$, or even an element in the subdifferential $D_t^-\Ch_t(u_t)$, which is at most single-valued since each $\Ch_t$ is a quadratic form.

Fortunately, by Theorem 2.12 in \cite{sturm2016} it turns out that at least for a.e. $t\in (0,T)$ $u_t\in\Dom(\Delta_t)$, thus $\Delta_tu_t\in \mathcal H$. Consequently
we may identify the gradient flow for Cheeger's energy with the heat flow. This is the statement of the following  theorem, for which we give an alternative proof using the dynamic EVI-property. In particular the heat flow can be constructed by the minimizing movement scheme from Section \ref{sechilbert}.
\begin{thm}\label{idi}
Let $(X,d_t,m_t)_{t\in[0,T]}$ be a family of Polish spaces with complete geodesic distances $d_t$ satisfying \eqref{loglip} such that
$m_t=e^{-f_t}m$, where $m\in\mathcal{P}(X)$ and $f_t$ are bounded functions satisfying and \eqref{lipschitzinspace}.
Assume that each static space is infinitesimally Hilbertian.
Let $\tilde u_t$ be the solution to the heat equation $\partial_t\tilde u_t=\Delta_t\tilde u_t$ on $(0,T)\times X$ starting in some
$\bar u \in \Dom(\Ch)$. Then $\tilde u_t$ satisfies
 \begin{align*}
  \partial_t {\tilde u}_t\in-D^-_t\Ch_t(\tilde u_t)\quad \text{ for a.e. }t\in(0,T),
 \end{align*}
 and can be constructed as the limit of a minimizing movement scheme. Conversely, let $u_t$ be the solution of the gradient flow
 of Cheeger's energy $\Ch_t$. Then $u_t$ solves the heat equation 
 \begin{align*}
  \partial_tu_t=\Delta_tu_t\, \text{ on }(0,T)\times X.
 \end{align*}
In particular $u_t=\tilde u_t$ in $L^2(X,m)$ for every $t\geq0$.
\end{thm}
\begin{proof}
 Both flows satisfy the dynamic EVI$(-L/2,\infty)$ gradient flow inequality almost everywhere by virtue of Proposition \ref{blabla} and Theorem 2.16 in 
 \cite{sturm2016}.
 Hence from the contraction estimate \eqref{almostcontraction}
 \begin{align*}
  ||u_t-\tilde u_t||_t^2\leq e^{7L(t-s)}||u_s-\tilde u_s||^2_s\quad \text{ for a.e. }t\geq s,
 \end{align*}
we obtain 
\begin{align*}
 ||u_t-\tilde u_t||^2_t\leq \lim_{s\to 0}e^{7L(t-s)}||u_s-\tilde u_s||_s=0\quad \text{ for a.e. }t,
\end{align*}
and hence by continuity $||u_t-\tilde u_t||_t=0$ for every $t$.
\end{proof}

\newpage
\bibliography{lit}

\begin{thebibliography}{10}

\bibitem{ag}
Luigi Ambrosio and Nicola Gigli.
\newblock A user's guide to optimal transport.
\newblock In {\em Modelling and optimisation of flows on networks}, pages
  1--155. Springer, Heidelberg, 2013.

\bibitem{ags}
Luigi Ambrosio, Nicola Gigli, and Giuseppe Savar{\'e}.
\newblock {\em {Gradient flows in metric spaces and in the Space of Probabiliy
  Measures}}.
\newblock Birkh\"auser, Basel, 2005.

\bibitem{agscalc}
Luigi Ambrosio, Nicola Gigli, and Giuseppe Savar{\'e}.
\newblock {Calculus and heat flow in metric measure spaces and applications to
  spaces with Ricci bounds from below}.
\newblock {\em Invent. Math.}, 195(2):289--391, 2013.

\bibitem{agsmet}
Luigi Ambrosio, Nicola Gigli, and Giuseppe Savar{\'e}.
\newblock {Metric measure spaces with Riemannian Ricci curvature bounded from
  below}.
\newblock {\em Duke Math. J.}, 163(7):1405--1490, 2014.

\bibitem{ams}
Luigi Ambrosio, Andrea Mondino, and Giuseppe Savar\'e.
\newblock {Nonlinear diffusion equations and curvature conditions in metric
  spaces}.
\newblock {\em arXiv:1509.07273}, 2015.

\bibitem{bogachev2007}
Vladimir~I. Bogachev.
\newblock {\em Measure Theory}, volume~1.
\newblock Springer-Verlag, Berlin, 2007.

\bibitem{cheeger}
Jeff Cheeger.
\newblock {Differentiability of Lipschitz functions on metric measure spaces}.
\newblock {\em Geom. Funct. Anal.}, 9(3):428--517, 1999.

\bibitem{erbar}
Matthias Erbar.
\newblock {The heat equation on manifolds as a gradient flow in the Wasserstein
  space}.
\newblock {\em Annales de l'I. H. P. Probabilit\'es et Statistiques},
  46(1):1--23, 2010.

\bibitem{eks2014}
Matthias Erbar, Kazumasa Kuwada, and Karl-Theodor Sturm.
\newblock {On the equivalence of the entropic curvature-dimension condition and
  Bochner's inequality on metric measure spaces}.
\newblock {\em Invent. Math.}, 201:993--1071, 2015.

\bibitem{ferreira}
Lucas Ferreira and Julio Valencia-Guevara.
\newblock {Gradient flows of time-dependent functionals in metric spaces and
  applications for PDE}.
\newblock {\em arXiv:1509.0416161v1}, 2015.

\bibitem{gigli2009}
Nicola Gigli.
\newblock {On the heat flow on metric measure spaces: existence, uniqueness and
  stability}.
\newblock {\em Calc. Var.}, 39(1-2):101--120, 2010.

\bibitem{gko}
Nicola Gigli, Kazumasa Kuwada, and Shin-ichi Ohta.
\newblock {Heat flow on Alexandrov spaces}.
\newblock {\em Comm. Pure Appl. Math.}, 66(3):307--331, 2013.

\bibitem{HN2015}
Robert Haslhofer and Aaron Naber.
\newblock {Weak solutions for the Ricci flow I}.
\newblock {\em arXiv:1504.00911}, 2015.

\bibitem{otto}
Richard Jordan, David Kinderlehrer, and Felix Otto.
\newblock {The variational formulation of the Fokker-Planck equation}.
\newblock {\em SIAM J. Math. Anal.}, 29:1--17, 1998.

\bibitem{sturm2016}
Eva Kopfer and Karl-Theodor Sturm.
\newblock {Heat flows on Time-dependent Metric Measure Spaces and Super-Ricci
  Flows}.
\newblock {\em arXiv:1611.02570}, 2017.

\bibitem{lierl2015}
Janna Lierl and Laurent Saloff-Coste.
\newblock {Parabolic Harnack inequality for time-dependent non-symmetric
  Dirichlet forms}.
\newblock {\em arXiv:1205.6493}, 2012.

\bibitem{ls}
John Lott and C{\'e}dric Villani.
\newblock Ricci curvature for metric-measure spaces via optimal transport.
\newblock {\em Annals of Mathematics}, pages 903--991, 2009.

\bibitem{mccanntopping}
Robert McCann and Peter Topping.
\newblock {Ricci flow, entropy and optimal transportation}.
\newblock {\em American Journal of Mathematics}, 132(3):711--730, 2010.

\bibitem{ohta/sturm}
Shin-ichi Ohta and Karl-Theodor Sturm.
\newblock {Heat flow on Finsler manifolds}.
\newblock {\em Comm. Pure Appl. Math.}, 62(11):1386--1433, 2009.

\bibitem{renardy2006introduction}
Michael Renardy and Robert~C. Rogers.
\newblock {\em An introduction to partial differential equations}, volume~13.
\newblock Springer-Verlag, New York, 2004.

\bibitem{rms}
Riccarda Rossi, Alexander Mielke, and Giuseppe Savar\'e.
\newblock {A metric approach to a class of doubly nonlinear evolution equations
  and applications}.
\newblock {\em Ann. Sc. Norm. Super. Pisa Cl. Sci.}, 7:97--169, 2008.

\bibitem{rossi}
Riccarda Rossi and Giuseppe Savar\'e.
\newblock {Gradient flows of non convex functionals in Hilbert spaces}.
\newblock {\em ESAIM Control Optim. Calc. Var.}, 12:564--614, 2006.

\bibitem{sturm2006}
Karl-Theodor Sturm.
\newblock {On the geometry of metric measure spaces. I}.
\newblock {\em Acta Math.}, 169(1):65--131, 2006.

\bibitem{sturm2015}
Karl-Theodor Sturm.
\newblock {Super Ricci flows for metric measure spaces. I}.
\newblock {\em arXiv:1603.02193}, 2016.

\bibitem{topping2006lectures}
Peter Topping.
\newblock {\em {Lectures on the Ricci flow}}, volume 325.
\newblock Cambridge University Press, 2006.

\bibitem{villani2009}
C\'edric Villani.
\newblock {\em Optimal transport, old and new}.
\newblock Springer-Verlag, Berlin, Heidelberg, 2009.

\bibitem{sturmrenesse}
Max-Konstantin von Renesse and Karl-Theodor Sturm.
\newblock {Transport inequalities, gradient estimates, entropy and Ricci
  curvature}.
\newblock {\em Comm. Pure and Appl. Math.}, 58(7):923--940, 2005.

\end{thebibliography}

\end{document}